%% file: breaking_multivariate.tex
\documentclass[11pt,reqno,tbtags,draft]{amsart}
\usepackage{amssymb}
\usepackage{url}
\usepackage[square,numbers]{natbib}
\usepackage{bm}
\bibpunct[, ]{[}{]}{;}{n}{,}{,}

\usepackage{mathtools}
\usepackage[rgb]{xcolor}
\xdefinecolor{dgreen}{rgb}{0, 0.65, 0} 
\xdefinecolor{dblue}{rgb}{0, 0, 0.75} 

\title[Breaking Multivariate Records]
{Breaking Multivariate Records}

\newcommand\urladdrx[1]{{\urladdr{\def~{{\tiny$\sim$}}#1}}}
\author{James Allen Fill}
\address{Department of Applied Mathematics and Statistics,
The Johns Hopkins University,
3400 N.~Charles Street,
Baltimore, MD 21218-2682 USA}
\email{jimfill@jhu.edu}
\urladdrx{http://www.ams.jhu.edu/~fill/}
\thanks{Research supported by
the Acheson~J.~Duncan Fund for the Advancement of Research in
Statistics.}

\numberwithin{equation}{section}

\renewcommand\le{\leqslant}
\renewcommand\ge{\geqslant}

\allowdisplaybreaks

\theoremstyle{plain}
\newtheorem{theorem}{Theorem}[section]
\newtheorem{lemma}[theorem]{Lemma}
\newtheorem{proposition}[theorem]{Proposition}
\newtheorem{corollary}[theorem]{Corollary}

\theoremstyle{definition}

\newtheorem{definition}[theorem]{Definition}

\newtheorem{remark}[theorem]{Remark}

\newtheorem*{acks}{Acknowledgements}

\theoremstyle{remark}

\newenvironment{romenumerate}[1][-10pt]{
\addtolength{\leftmargini}{#1}\begin{enumerate}
 \renewcommand{\labelenumi}{\textup{(\roman{enumi})}}%
 \renewcommand{\theenumi}{\textup{(\roman{enumi})}}%
 }{\end{enumerate}}

\newcounter{oldenumi}
{\setcounter{oldenumi}{\value{enumi}}
\begin{romenumerate} \setcounter{enumi}{\value{oldenumi}}}
{\end{romenumerate}}

\newcounter{thmenumerate}

\newcounter{xenumerate}   

\newcommand\xfootnote[1]{\unskip\footnote{#1}$ $} 

\newcommand\pfitem[1]{\par(#1):}
\newcommand\pfitemx[1]{\par#1:}
\newcommand\pfitemref[1]{\pfitemx{\ref{#1}}}

\newcommand\pfcase[2]{\smallskip\noindent\emph{Case #1: #2} \noindent}
\newcommand\step[2]{\smallskip\noindent\emph{Step #1: #2} \noindent}
\newcommand\stepx{\smallskip\noindent\refstepcounter{steps}%
 \emph{Step \arabic{steps}:}\noindent}

\newcommand{\refT}[1]{Theorem~\ref{#1}}
\newcommand{\refC}[1]{Corollary~\ref{#1}}
\newcommand{\refL}[1]{Lemma~\ref{#1}}
\newcommand{\refR}[1]{Remark~\ref{#1}}
\newcommand{\refS}[1]{Section~\ref{#1}}
\newcommand{\refSS}[1]{Subsection~\ref{#1}}
\newcommand{\refP}[1]{Proposition~\ref{#1}}
\newcommand{\refD}[1]{Definition~\ref{#1}}
\newcommand{\refE}[1]{Example~\ref{#1}}
\newcommand{\refF}[1]{Figure~\ref{#1}}
\newcommand{\refApp}[1]{Appendix~\ref{#1}}

\newcommand{\refTab}[1]{Table~\ref{#1}}
\newcommand{\refand}[2]{\ref{#1} and~\ref{#2}}

\newcommand\marginal[1]{\marginpar{\raggedright\parindent=0pt\tiny #1}}
\newcommand\JF{\marginal{JF}}

\newcommand\kolla{\marginal{CHECK! SJ}}
\newcommand\ms[1]{\texttt{[ms #1]}}
\newcommand\XXX{XXX \marginal{XXX}}
\newcommand\REM[1]{{\raggedright\texttt{[#1]}\par\marginal{XXX}}}
\newcommand\rem[1]{{\texttt{[#1]}\marginal{XXX}}}

\newcommand\linebreakx{\unskip\marginal{$\backslash$linebreak}\linebreak}

\begingroup
  \divide\count255 by 60
  \multiply\count255 by -60
  \advance\count255 by \time
  \ifnum \count255 < 10 \xdef\klockan{\the\count1.0\the\count255}
  \else\xdef\klockan{\the\count1.\the\count255}\fi
\endgroup

\newcommand\nopf{\qed}   
\newcommand\noqed{\renewcommand{\qed}{}} 
\newcommand\qedtag{\eqno{\qed}}

\DeclareMathOperator*{\sumx}{\sum\nolimits^{*}}
\DeclareMathOperator*{\sumxx}{\sum\nolimits^{**}}
\newcommand{\sumio}{\sum_{i=0}^\infty}
\newcommand{\sumjo}{\sum_{j=0}^\infty}
\newcommand{\sumj}{\sum_{j=1}^\infty}
\newcommand{\sumko}{\sum_{k=0}^\infty}
\newcommand{\sumk}{\sum_{k=1}^\infty}
\newcommand{\summo}{\sum_{m=0}^\infty}
\newcommand{\sumno}{\sum_{n=0}^\infty}
\newcommand{\sumn}{\sum_{n=1}^\infty}
\newcommand{\sumin}{\sum_{i=1}^n}
\newcommand{\sumjn}{\sum_{j=1}^n}
\newcommand{\sumkn}{\sum_{k=1}^n}
\newcommand{\prodin}{\prod_{i=1}^n}

\newcommand\set[1]{\ensuremath{\{#1\}}}
\newcommand\bigset[1]{\ensuremath{\bigl\{#1\bigr\}}}
\newcommand\Bigset[1]{\ensuremath{\Bigl\{#1\Bigr\}}}
\newcommand\biggset[1]{\ensuremath{\biggl\{#1\biggr\}}}
\newcommand\lrset[1]{\ensuremath{\left\{#1\right\}}}
\newcommand\xpar[1]{(#1)}
\newcommand\bigpar[1]{\bigl(#1\bigr)}
\newcommand\Bigpar[1]{\Bigl(#1\Bigr)}
\newcommand\biggpar[1]{\biggl(#1\biggr)}
\newcommand\lrpar[1]{\left(#1\right)}
\newcommand\bigsqpar[1]{\bigl[#1\bigr]}
\newcommand\Bigsqpar[1]{\Bigl[#1\Bigr]}
\newcommand\biggsqpar[1]{\biggl[#1\biggr]}
\newcommand\lrsqpar[1]{\left[#1\right]}
\newcommand\xcpar[1]{\{#1\}}
\newcommand\bigcpar[1]{\bigl\{#1\bigr\}}
\newcommand\Bigcpar[1]{\Bigl\{#1\Bigr\}}
\newcommand\biggcpar[1]{\biggl\{#1\biggr\}}
\newcommand\lrcpar[1]{\left\{#1\right\}}
\newcommand\abs[1]{|#1|}
\newcommand\bigabs[1]{\bigl|#1\bigr|}
\newcommand\Bigabs[1]{\Bigl|#1\Bigr|}
\newcommand\biggabs[1]{\biggl|#1\biggr|}
\newcommand\lrabs[1]{\left|#1\right|}
\def\rompar(#1){\textup(#1\textup)}    
\newcommand\xfrac[2]{#1/#2}
\newcommand\xpfrac[2]{(#1)/#2}
\newcommand\xqfrac[2]{#1/(#2)}
\newcommand\xpqfrac[2]{(#1)/(#2)}
\newcommand\parfrac[2]{\lrpar{\frac{#1}{#2}}}
\newcommand\bigparfrac[2]{\bigpar{\frac{#1}{#2}}}
\newcommand\Bigparfrac[2]{\Bigpar{\frac{#1}{#2}}}
\newcommand\biggparfrac[2]{\biggpar{\frac{#1}{#2}}}
\newcommand\xparfrac[2]{\xpar{\xfrac{#1}{#2}}}
\newcommand\innprod[1]{\langle#1\rangle}
\newcommand\expbig[1]{\exp\bigl(#1\bigr)}
\newcommand\expBig[1]{\exp\Bigl(#1\Bigr)}
\newcommand\explr[1]{\exp\left(#1\right)}
\newcommand\expQ[1]{e^{#1}}
\def\xexp(#1){e^{#1}}
\newcommand\ceil[1]{\lceil#1\rceil}
\newcommand\floor[1]{\lfloor#1\rfloor}
\newcommand\lrfloor[1]{\left\lfloor#1\right\rfloor}
\newcommand\frax[1]{\{#1\}}
\newcommand\setn{\set{1,\dots,n}}
\newcommand\nn{[n]}
\newcommand\ntoo{\ensuremath{{n\to\infty}}}
\newcommand\Ntoo{\ensuremath{{N\to\infty}}}
\newcommand\asntoo{\text{as }\ntoo}
\newcommand\ktoo{\ensuremath{{k\to\infty}}}
\newcommand\mtoo{\ensuremath{{m\to\infty}}}
\newcommand\stoo{\ensuremath{{s\to\infty}}}
\newcommand\ttoo{\ensuremath{{t\to\infty}}}
\newcommand\xtoo{\ensuremath{{x\to\infty}}}
\newcommand\bmin{\wedge}
\newcommand\norm[1]{\|#1\|}
\newcommand\bignorm[1]{\bigl\|#1\bigr\|}
\newcommand\Bignorm[1]{\Bigl\|#1\Bigr\|}
\newcommand\downto{\searrow}
\newcommand\upto{\nearrow}
\newcommand\half{\tfrac12}
\newcommand\thalf{\tfrac12}

\newcommand\punkt{.\spacefactor=1000}    
\newcommand\iid{i.i.d\punkt}    
\newcommand\ie{i.e\punkt}
\newcommand\eg{e.g\punkt}
\newcommand\viz{viz\punkt}
\newcommand\cf{cf\punkt}
\newcommand{\as}{a.s\punkt}
\newcommand{\aex}{a.e\punkt}
\newcommand{\io}{i.o\punkt}
\renewcommand{\ae}{\vu}  
\newcommand\whp{w.h.p\punkt}

\newcommand\ii{\mathrm{i}}

\newcommand{\tend}{\longrightarrow}
\newcommand\Lcto{\overset{\mathcal{L}}{\tend}}
\newcommand\dto{\overset{\mathrm{d}}{\tend}}
\newcommand\pto{\overset{\mathrm{p}}{\tend}}
\newcommand\Pto{\overset{\mathrm{P}}{\tend}}
\newcommand\asto{\overset{\mathrm{a.s.}}{\tend}}
\newcommand\eqd{\overset{\mathrm{d}}{=}}
\newcommand\neqd{\overset{\mathrm{d}}{\neq}}
\newcommand\op{o_{\mathrm p}}
\newcommand\Op{O_{\mathrm p}}

\newcommand\bbR{\mathbb R}
\newcommand\bbC{\mathbb C}
\newcommand\bbN{\mathbb N}
\newcommand\bbT{\mathbb T}
\newcommand\bbQ{\mathbb Q}
\newcommand\bbZ{\mathbb Z}
\newcommand\bbZleo{\mathbb Z_{\le0}}
\newcommand\bbZgeo{\mathbb Z_{\ge0}}

\newcounter{CC}
\newcommand{\CC}{\stepcounter{CC}\CCx} 
\newcommand{\CCx}{C_{\arabic{CC}}}     
\newcommand{\CCdef}[1]{\xdef#1{\CCx}}     
\newcommand{\CCname}[1]{\CC\CCdef{#1}}    
\newcommand{\CCreset}{\setcounter{CC}0} 
\newcounter{cc}
\newcommand{\cc}{\stepcounter{cc}\ccx} 
\newcommand{\ccx}{c_{\arabic{cc}}}     
\newcommand{\ccdef}[1]{\xdef#1{\ccx}}     
\newcommand{\ccname}[1]{\cc\ccdef{#1}}    
\newcommand{\ccreset}{\setcounter{cc}0} 

\renewcommand\Re{\operatorname{Re}}
\renewcommand\Im{\operatorname{Im}}

\newcommand\E{\operatorname{\mathbb E{}}}
\renewcommand\P{\operatorname{\mathbb P{}}}
\renewcommand\L{\operatorname{L}}
\newcommand\Var{\operatorname{Var}}
\newcommand\Cov{\operatorname{Cov}}
\newcommand\Corr{\operatorname{Corr}}
\newcommand\Exp{\operatorname{Exp}}
\newcommand\Po{\operatorname{Po}}
\newcommand\Bi{\operatorname{Bi}}
\newcommand\Bin{\operatorname{Bin}}
\newcommand\Be{\operatorname{Be}}
\newcommand\Ge{\operatorname{Ge}}
\newcommand\NBi{\operatorname{NegBin}}
\newcommand\Res{\operatorname{Res}}
\newcommand\fall[1]{^{\underline{#1}}}
\newcommand\rise[1]{^{\overline{#1}}}

\newcommand\supp{\operatorname{supp}}
\newcommand\sgn{\operatorname{sgn}}
\newcommand\Tr{\operatorname{Tr}}

\newcommand\ga{\alpha}
\newcommand\gb{\beta}
\newcommand\gd{\delta}
\newcommand\gD{\Delta}
\newcommand\gee{\epsilon}
\newcommand\gf{\varphi}
\newcommand\gam{\gamma}
\newcommand\gG{\Gamma}
\newcommand\gk{\kappa}
\newcommand\gl{\lambda}
\newcommand\gL{\Lambda}
\newcommand\go{\omega}
\newcommand\gO{\Omega}
\newcommand\gs{\sigma}
\newcommand\gss{\sigma^2}
\newcommand\gth{\theta}
\newcommand\eps{\varepsilon}
\newcommand\ep{\varepsilon}

\newcommand\bJ{\bar J}

\newcommand\cA{\mathcal A}
\newcommand\cB{\mathcal B}
\newcommand\cC{\mathcal C}
\newcommand\cD{\mathcal D}
\newcommand\cE{\mathcal E}
\newcommand\cF{\mathcal F}
\newcommand\cG{\mathcal G}
\newcommand\cH{\mathcal H}
\newcommand\cI{\mathcal I}
\newcommand\cJ{\mathcal J}
\newcommand\cK{\mathcal K}
\newcommand{\tcK}{\widetilde{\cK}}
\newcommand\cL{{\mathcal L}}
\newcommand\cM{\mathcal M}
\newcommand\cN{\mathcal N}
\newcommand\cO{\mathcal O}
\newcommand\cP{\mathcal P}
\newcommand\cQ{\mathcal Q}
\newcommand\cR{{\mathcal R}}
\newcommand\cS{{\mathcal S}}
\newcommand\cT{{\mathcal T}}
\newcommand\cU{{\mathcal U}}
\newcommand\cV{\mathcal V}
\newcommand\cW{\mathcal W}
\newcommand\cX{{\mathcal X}}
\newcommand\cY{{\mathcal Y}}
\newcommand\cZ{{\mathcal Z}}

\newcommand\ett[1]{\boldsymbol1_{#1}} 
\newcommand\bigett[1]{\boldsymbol1\bigcpar{#1}} 
\newcommand\Bigett[1]{\boldsymbol1\Bigcpar{#1}} 
\newcommand\etta{\boldsymbol1} 

\newcommand\smatrixx[1]{\left(\begin{smallmatrix}#1\end{smallmatrix}\right)}

\newcommand\limn{\lim_{n\to\infty}}
\newcommand\limN{\lim_{N\to\infty}}

\newcommand\qw{^{-1}}
\newcommand\qww{^{-2}}
\newcommand\qq{^{1/2}}
\newcommand\qqw{^{-1/2}}
\newcommand\qqq{^{1/3}}
\newcommand\qqqb{^{2/3}}
\newcommand\qqqw{^{-1/3}}
\newcommand\qqqbw{^{-2/3}}
\newcommand\qqqq{^{1/4}}
\newcommand\qqqqc{^{3/4}}
\newcommand\qqqqw{^{-1/4}}
\newcommand\qqqqcw{^{-3/4}}

\newcommand\intii{\int_{-1}^1}
\newcommand\intoi{\int_0^1}
\newcommand\intoo{\int_0^\infty}
\newcommand\intoooo{\int_{-\infty}^\infty}
\newcommand\oi{[0,1]}
\newcommand\ooo{[0,\infty)}
\newcommand\ooox{[0,\infty]}
\newcommand\oooo{(-\infty,\infty)}
\newcommand\setoi{\set{0,1}}

\newcommand\dtv{d_{\mathrm{TV}}}

\newcommand\dd{\,\mathrm{d}}
\newcommand\ddx{\mathrm{d}}

\newcommand{\pgf}{probability generating function}
\newcommand{\mgf}{moment generating function}
\newcommand{\chf}{characteristic function}
\newcommand{\ui}{uniformly integrable}
\newcommand\rv{random variable}
\newcommand\lhs{left-hand side}
\newcommand\rhs{right-hand side}

\newcommand\gnp{\ensuremath{G(n,p)}}
\newcommand\gnm{\ensuremath{G(n,m)}}
\newcommand\gnd{\ensuremath{G(n,d)}}
\newcommand\gnx[1]{\ensuremath{G(n,#1)}}

\newcommand\etto{\bigpar{1+o(1)}}
\newcommand\GW{Galton--Watson}
\newcommand\GWt{\GW{} tree}
\newcommand\cGWt{conditioned \GW{} tree}
\newcommand\GWp{\GW{} process}
\newcommand\tX{{\widetilde X}}
\newcommand\tY{{\widetilde Y}}
\newcommand\kk{\varkappa}
\newcommand\spann[1]{\operatorname{span}(#1)}
\newcommand\tn{\cT_n}
\newcommand\tnv{\cT_{n,v}}
\newcommand\rea{\Re\ga}
\newcommand\wgay{{-\ga-\frac12}}
\newcommand\qgay{{\ga+\frac12}}

\newcommand\uu{\mathbf u}

\newcommand\xx{\mathbf x}
\newcommand\yy{\mathbf y}
\newcommand\zz{\mathbf z}
\newcommand\ww{\mathbf w}
\newcommand\UU{\mathbf U}
\newcommand\XX{\mathbf X}
\newcommand\YY{\mathbf Y}
\newcommand\ZZ{\mathbf Z}
\newcommand\gdgd{\bm \gd}
\newcommand{\gDgD}{\bm \gD}
\newcommand\gege{\bm \gee}
\newcommand\hX{\hat X}
\newcommand\sgt{simply generated tree}
\newcommand\sgrt{simply generated random tree}
\newcommand\hh[1]{d(#1)}
\newcommand\WW{\widehat W}
\newcommand\coi{C\oi}
\newcommand\out{\gd^+}
\newcommand\zne{Z_{n,\eps}}
\newcommand\ze{Z_{\eps}}
\newcommand\gatoo{\ensuremath{\ga\to\infty}}
\newcommand\rtoo{\ensuremath{r\to\infty}}
\newcommand\Yoo{Y_\infty}
\newcommand\bes{R}
\newcommand\tex{\tilde{\ex}}
\newcommand\tbes{\tilde{\bes}}
\newcommand\Woo{W_\infty}
\newcommand{\thm}{\tilde m_1}
\newcommand{\bbb}{B^{(3)}}
\newcommand{\rr}{r^{1/2}}
\newcommand\coo{C[0,\infty)}
\newcommand\coT{\ensuremath{C[0,T]}}
\newcommand\expx[1]{e^{#1}}
\newcommand\gdtau{\gD\tau}
\newcommand\ygam{Y_{(\gam)}}
\newcommand\EE{V}
\newcommand\pigsqq{\sqrt{2\pi\gss}}
\newcommand\pigsqqw{\frac{1}{\sqrt{2\pi\gss}}}
\newcommand\gapigsqqw{\frac{(\ga-\frac12)\qw}{\sqrt{2\pi\gss}}}
\newcommand\gdd{\frac{\gd}{2}}

\newcommand\raisetagbase{\raisetag{\baselineskip}}

\newcommand\eit{e^{\ii t}}
\newcommand\emit{e^{-\ii t}}
\newcommand\tgf{\tilde\gf}
\newcommand\txi{\tilde\xi}
\newcommand\intT{\frac{1}{2\pi}\int_{-\pi}^\pi}
\newcommand\intpi{\int_{-\pi}^\pi}
\newcommand\Li{\operatorname{Li}}
\newcommand\doi{D_{01}}
\newcommand\cHoi{\cH(\doi)}
\newcommand\db{D}
\newcommand\dbm{D_-}
\newcommand\dbmb{\overline D_-}
\newcommand\dbmx{\widehat D_-}
\newcommand\bdb{\overline{D_B}}
\newcommand\xq{\setminus\set{\frac12}}
\newcommand\xo{\setminus\set{0}}
\newcommand\tqn{t/\sqrt n}
\newcommand\intpm[1]{\int_{-#1}^{#1}}
\newcommand\gnaxt{g_n(\ga,x,t)}
\newcommand\gaxt{g(\ga,x,t)}
\newcommand\gssx{\frac{\gss}2}
\newcommand\tf{\tilde f}

\newcommand\tq{\tilde q}
\newcommand\tr{\tilde r}

\newcommand\gao{\ga_0}
\newcommand\ppp{\cP_1}
\newcommand\dx{D^*}
\newcommand\tpsi{\tilde\psi}
\newcommand\tgD{\tilde\gD}
\newcommand\xinn{\xi_{n-1,N}}
\newcommand\zzn{\frac12+\ii y_n}
\newcommand\OHD{O_{\cH(D)}}
\newcommand\OHDx{O_{\cH(\dx)}}
\newcommand\tgdn{\tgD_N}
\newcommand\xgdn{\gD^*_N}
\newcommand\intt{\int_0^T}
\newcommand\act{|\cT|}

\newcommand{\ignore}[1]{}
\newcommand{\Holder}{H\"older}
\newcommand\CS{Cauchy--Schwarz}
\newcommand\CSineq{\CS{} inequality}
\newcommand{\Levy}{L\'evy}
\newcommand{\Takacs}{Tak\'acs}
\newcommand{\Frechet}{Fr\'echet}

\newcommand{\maple}{\texttt{Maple}}

\newcommand\citex{\REM}
\newcommand\refx[1]{\texttt{[#1]}}
\newcommand\xref[1]{\texttt{(#1)}}

\usepackage{tikz}
\usepackage{amssymb}
\usetikzlibrary{arrows,positioning}
\usepackage[english]{babel}
\usepackage{pgfplotstable}
\usepackage{pgfplots}
\usepackage{adjustbox}
\usetikzlibrary{calc}
\usepackage{relsize}
\usetikzlibrary{arrows.meta}

\pgfdeclarelayer{background layer}
\pgfdeclarelayer{foreground layer}
\pgfsetlayers{background layer,main,foreground layer}

\tikzset{>={Latex[width=5mm,length=5mm]}}

\pgfplotsset{compat=1.3}

\begin{document}

\date{September~29, 2021; revised March~7, 2023}

\begin{abstract}

For a sequence of \iid\ $d$-dimensional random vectors with independent continuously distributed coordinates, say that the $n$th observation in the sequence sets a record if it is not dominated in every coordinate by an earlier observation; for $j \leq n$, say that the $j$th observation is a current record at time~$n$ if it has not been dominated in every coordinate by any of the first~$n$ observations; and say that the $n$th observation breaks~$k$ records if it sets a record and there are $k$ observations that are current records at time $n - 1$ but not at time~$n$.

For general dimension~$d$, we identify, with proof, the asymptotic conditional distribution of the number of records broken by an observation given that the observation sets a record.

Fix~$d$, and let $\cK(d)$ be a random variable with this distribution.  We show that the (right) tail of $\cK(d)$ satisfies
\[
\P(\cK(d) \geq k) \leq \exp\left[ - \Omega\!\left( k^{(d - 1) / (d^2 + d - 3)} \right) \right]\mbox{\ \ as $k \to \infty$}
\]
and
\[
\P(\cK(d) \geq k) \geq \exp\left[ - O\!\left( k^{1 / (d - 1)} \right) \right]\mbox{\ \ as $k \to \infty$}.
\]

When $d = 2$, the description of $\cK(2)$ in terms of a Poisson process agrees with the main result from \cite{Fillbreaking(2021)} that $\cK(2)$ has the same distribution as $\cG - 1$, where 
$\cG \sim \mbox{\rm Geometric$(1/2)$}$.  Note that the lower bound on $\P(\cK(d) \geq k)$ implies that the distribution of 
$\cK(d)$ is \emph{not} (shifted) Geometric for any $d \geq 3$.

We show that $\P(\cK(d) \geq 1) = \exp[-\Theta(d)]$ as $d \to \infty$; in particular, $\cK(d) \to 0$ in probability as $d \to \infty$.
\end{abstract}

\maketitle

\section{Introduction and main result}
\label{S:intro}

Let $\XX^{(1)}, \XX^{(2)}, \dots$ be \iid\ (independent and identically distributed) copies of a random 
vector $\XX = (X_1, \ldots, X_d)$ with independent coordinates, each uniformly distributed over the unit interval.  In this paper, for general dimension~$d$, we prove (\refT{T:main}) that the conditional distribution of the number of records broken by the $n$th observation $\XX^{(n)}$ given that $\XX^{(n)}$ sets a record has a weak limit [call it $\mu(d)$, the distribution or law $\cL(\cK(d))$ of a random variable $\cK(d)$] as $n \to \infty$, and we identify this limit.  (See \refD{D:record} for the relevant definitions regarding records.)  The case $d = 1$ is trivial: 
$\cK(d) = 1$ with probability~$1$.  The case $d = 2$ is treated in depth in \cite{Fillbreaking(2021)}, where it is shown that $\mu(2) = \cL(\cG -1)$ with $\cG \sim \mbox{\rm Geometric$(1/2)$}$.
\medskip 

{\bf Notation.\ }We let ${\bf 1}(E) = \mbox{$1$ or $0$}$ according as~$E$ is true or false, and for a random variable~$Y$ and an event~$A$ we write $\E(Y; A)$ as shorthand for $\E[Y {\bf 1}(A)]$.
We write $\ln$ or $\L$ for natural logarithm.
For $d$-dimensional vectors $\xx = (x_1, \dots, x_d)$ and $\yy = (y_1, \dots, y_d)$,
write $\xx \prec \yy$ 
to mean that $x_j < y_j$ 
for $j = 1, \dots, d$. 
The notation $\xx \succ \yy$ means $\yy \prec \xx$.
The notations $\xx \preceq \yy$ and $\yy \succeq \xx$ each mean that either $\xx \prec \yy$ or $\xx = \yy$.
Whenever we speak of incomparable vectors, we will mean vectors belonging to $\bbR^d$ that are (pairwise-)incomparable with respect to the partial order $\preceq$.
We write 
$x_+ := \sum_{j = 1}^d x_j$ 
for the sum 
of coordinates of $x = (x_1, \dots, x_d)$
and $\|\cdot\|$ for $\ell^1$-norm; thus $x_+ = \|\xx\|$ if $\xx \succ {\bf 0}$, where~${\bf 0}$ is the vector 
$(0, \dots, 0)$.
\medskip

We begin with some relevant definitions, taken 
from \cite{Fillboundary(2020), Fillgenerating(2018), Fillbreaking(2021)}.  We give definitions for 
record-\emph{large} values; we consider such records exclusively except in the case of \refT{T:main}$'$, which deals with record-\emph{small} values (for which the analogous definitions are obvious). 

Let $\XX^{(1)}, \XX^{(2)}, \dots$ be \iid\ (independent and identically distributed) copies of a random 
vector~$\XX$ with independent coordinates, each drawn from a common specified distribution.  We mainly consider the case (as in \cite{Fillboundary(2020), Fillgenerating(2018)}) where the common distribution is Exponential$(1)$, but in \refT{T:main}$'$ we consider the uniform distribution over the unit interval.  As far as \emph{counting} records or broken records, any continuous distribution gives the same results.   

\begin{definition}
\label{D:record}
(a)~We say that $\XX^{(n)}$ is a \emph{Pareto record} (or simply \emph{record}, or that $\XX^{(n)}$ \emph{sets} a record at time~$n$) if $\XX^{(n)} \not\prec \XX^{(i)}$ for all $1 \leq i < n$.

(b)~If $1 \leq j \leq n$, we say that $\XX^{(j)}$ is a \emph{current record} (or \emph{remaining record}, or \emph{maximum}) at time~$n$ if $\XX^{(j)} \not\prec \XX^{(i)}$ for all $i \in [n]$. 

(c)~If $0 \leq k \leq n$, we say that $\XX^{(n)}$ \emph{breaks} (or \emph{kills}) $k$ records if $\XX^{(n)}$ sets a record and there exist precisely~$k$ values~$j$ with $1 \leq j < n$ such that $\XX^{(j)}$ is a current record at time $n - 1$ but is not a current record at time~$n$.

(d)~For $n \geq 1$ (or $n \geq 0$, with the obvious conventions) let $R_n$ denote the number of records 
$\XX^{(k)}$ with $1 \leq k \leq n$, and let $r_n$ denote the number of remaining records at time~$n$.
\end{definition}

\begin{remark}
\refD{D:record} is illustrated in Figures \ref{F:figure}--\ref{F:frontier_3d}, adapted from \cite[Figure~1]{Fillbreaking(2021)} and 
\cite[Figure~2]{Fillgenerating(2018)}, respectively.
In dimension~$2$, 
remaining records 
can be ordered northwest to southeast, as seen in \refF{F:figure}.
\begin{figure}[htb]
\scalebox{1.0}{\input{breaking_figure}}
\caption{In this $2$-dimensional example, after $n - 1$ observations, none of which fall in the shaded region,
there are $r_{n - 1} = 6$ remaining records.
The $n^{\rm \scriptsize th}$ observation, shown in bright green at the intersection of the dashed lines,
breaks the $K_n = k = 3$ remaining records shown in red
but not the $r_{n  - 1} - K_n = 3$ remaining records shown in dark green.}
\ignore{
\caption{In this example, after $n - 1$ observations, none of which fall in the  
shaded region,
there are 
$r_{n - 1} = 6$ remaining 
\emph{small}-records. 
The 
$n^{\rm \scriptsize th}$ observation, shown in 
bright green,
breaks the $K_{n - 1} = k = 3$ remaining records shown in red 
but not the 
$r_{n  - 1}- K_{n - 1} = 3$ remaining records shown in dark green.}
}
\label{F:figure}
\end{figure}
In dimensions $d \geq 3$, the 
record-setting region 
is a (typically) 
more complicated set,
as illustrated in \refF{F:frontier_3d}.
\begin{figure}[htb]
\scalebox{1.0}{\input{figure1}}
\ignore{
\caption{Example of a 
record frontier 
in dimension $d=3$ with  
$8$ remaining records shown in green 
and the resulting 
$17$ generators shown in 
violet.
The lower boundary of 
one of the orthants 
$O^+_g$ is shown using 
dashed lines.}
}
\caption{In this $3$-dimensional example, suppose that there are $8$ remaining records, as shown in green.  A new observation will
set a record if and only if it belongs to the union of positive orthants translated by one of the $17$ points shown in purple.  The lower boundary of one of these translated orthants (corresponding to the purple point labeled~$g$) is shown using dashed lines.}
\label{F:frontier_3d}
\end{figure}
\end{remark}

\ignore{
{\bf Is the definition of $\mbox{RS}_n$ needed?  If not, then delete.}
\begin{definition}
\label{D:RS}
The \emph{record-setting region} at time~$n$ is the (random) closed set of points
\[
\mbox{RS}_n := \{x \in [0, 1)^d: x \not\prec \XX^{(i)}\mbox{\ for all $i \in [n]$}\}.
\]
%
\end{definition}
}
A maximum~$\xx$ for a given set $S \subseteq \bbR^d$ is defined in similar fashion to \refD{D:record}(b): We say that~$\xx$ is a \emph{maximum} of~$S$ if $\xx \not\prec \yy$ for all $\yy \in S$.  In this language, if $1 \leq j \leq n$ we have that $\XX^{(j)}$ is a remaining record at time~$n$ if $\XX^{(j)}$ is a maximum of $\{ \XX^{(1)}, \XX^{(2)}, \ldots, \XX^{(n)} \}$.
\medskip

Here is the main result of this paper, illustrated in \refF{F:figure_poi}.
\begin{figure}[t]
           \centering
           \scalebox{1.3}{
                      \input{poisson}
           }
\caption{In this ($d = 2$) realization of the Poisson point process in \refT{T:main} and the theorem's possible extension in \refS{S:informal}, the maxima are shown in red [and labeled $\mathbf{Z}^{(1)}$ and $\mathbf{Z}^{(2)}$] if dominated by the bright green point~$\xx$ (so $\cK = 2$) and in dark green otherwise.  The intensity of the process vanishes in the shaded region.}           
\label{F:figure_poi}
\end{figure}
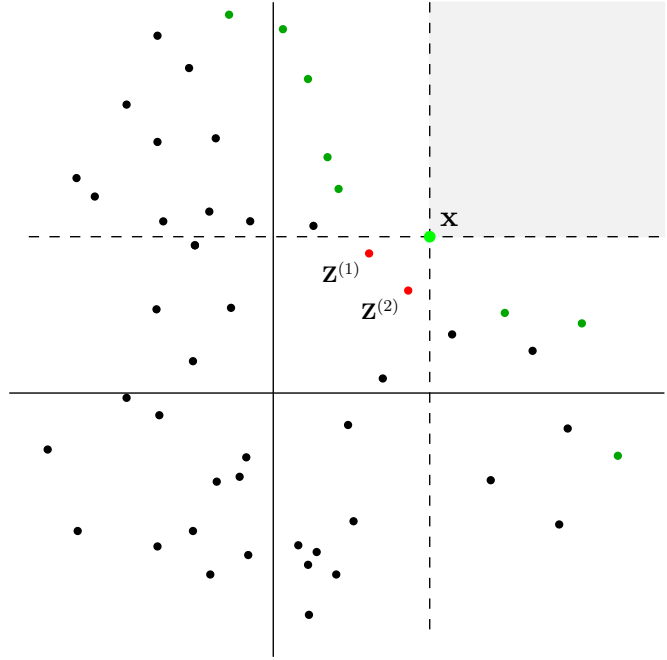

\begin{theorem}
\label{T:main}
Let $\XX^{(1)}, \XX^{(2)}, \dots$ be \iid\ $d$-variate observations, each with independent {\rm Exponential}$(1)$ coordinates.
Let $K_n = -1$ if $\XX^{(n)}$ does not set a record, and otherwise let $K_n$ denote the number of remaining records killed by $\XX^{(n)}$.  Then $K_n$, conditionally given $K_n \geq 0$, converges in distribution as $n \to \infty$ to the law of $\cK \equiv \cK(d)$, where 
\[
\mbox{\rm $\cL(\cK)$ is the mixture $\E \cL(\cK_G)$}.
\]
Here~$G$ is distributed {\rm standard Gumbel} and
\[
\cK_g = \mbox{\rm number of maxima dominated by $\xx \equiv \xx(g) := (g / d, \ldots, g / d)$}
\]
in a nonhomogeneous Poisson point process in $\bbR^d$ with intensity function
\[
{\bf 1}(\zz \not\succ \xx) e^{- z_+}, \quad \zz \in \bbR^d.
\]
\end{theorem}

It is easy to see from the form of the intensity function that, for each $g \in \bbR$, the distribution of $\cK_g$ is unchanged if $\xx(g)$ is replaced by any point $\tilde{\xx}(g) \in \bbR^d$ with $\tilde{x}(g)_+ = g$.
\smallskip

By using the decreasing bijection $z \mapsto e^{-z}$ from $\bbR$ to $\bbR_+ := (0, \infty)$, which is also a bijection from 
$\bbR_+$ to $(0, 1)$, \refT{T:main} can be recast equivalently as follows:
\medskip

\par\noindent
{\bf \refT{T:main}$'$.}\ \emph{Suppose that independent $d$-variate observations, each uniformly distributed in the unit hypercube $(0, 1)^d$, arrive at times $1, 2, \ldots$, and consider record-\emph{small} values.
Let $K_n = -1$ if the $n^{\rm \scriptsize th}$ observation is not a new record, and otherwise let $K_n$ denote the number of remaining records killed by the $n^{\rm \scriptsize th}$ observation.  Then $K_n$, conditionally given $K_n \geq 0$, converges in distribution as $n \to \infty$ to the law of $\cK \equiv \cK(d)$, where
\[
\mbox{\rm $\cL(\cK)$ is the mixture $\E \cL(\cK'_W)$}.
\]
Here~$W$ is distributed {\rm Exponential}$(1)$ and
\[
\cK'_w = \mbox{\rm number of minima that dominate $\xx' \equiv \xx'(w) := (w^{1 / d}, \ldots, w^{1 / d})$}
\]
in a (homogeneous) unit-rate Poisson point process in the set
\[
\{\zz \in \bbR^d_+:\zz \not\prec \xx'\}.
\]
}

\refS{S:informal} gives a non-rigorous but informative proof of \refT{T:main}, and the much longer \refS{S:formal} gives a rigorous proof.  Tail probabilities for $\cK(d)$ are studied (asymptotically) for each fixed~$d$ in \refS{S:tail}; moments are considered, too.  Asymptotics of $\cL(\cK(d))$ as $d \to \infty$ are studied in \refS{S:d to infinity}; in particular, we find that $\cK(d)$ converges in probability to~$0$.  Finally, in \refS{S:d=2} we show that the main \refT{T:main} of this paper reduces to the main Theorem~1.1 of \cite{Fillgenerating(2018)} when 
$d = 2$.

\section{Informal proof of main theorem}
\label{S:informal}

In this section we give an informal (heuristic) proof of \refT{T:main}; a formal proof is given in the next section.  The informal proof suggests that it should be possible to prove extensions of \refT{T:main} (under the same hypotheses as the theorem and using the same notation) such as the following, but we haven't pursued such extensions in this paper.
\medskip

\par\noindent
{\bf Possible extension.}\ \emph{Let $k \geq 1$.  Over the event $K_n = k$, let $\YY^{(n, 1)}, \ldots$, 
$\YY^{(n, k)}$ denote the values of the~$k$ records killed by $\XX^{(n)}$, listed (for definiteness) in increasing order of first coordinate, and let $\YY^{(n, \ell)} := \XX^{(n)}$ for $\ell > k$.  Then $(\XX^{(n)} - \YY^{(n, 1)}, \XX^{(n)} - \YY^{(n, 2)}, \dots)$ converges in distribution to 
\[
(\xx(G) - \ZZ^{(1)}(G), \xx(G) - \ZZ^{(2)}(G), \dots);
\] 
here, over the event $\{G = g,\,\cK = k\}$, the points $\ZZ^{(1)}(g)$, $\ldots, \ZZ^{(k)}(g)$ are the~$k$ maxima counted by 
$\cK_g$, listed in increasing order of first coordinate, and $\ZZ^{(\ell)}(g) := \xx(g)$ for $\ell > k$.} 
\medskip  

Before beginning our heuristic proof of \refT{T:main}, we gather and utilize important information from~\cite{Fillboundary(2020)} in the following remark. 

\begin{remark}
\label{R:Fillboundary}
(a) It is well known that
\begin{equation}
\label{pn}
p_n := \P(\XX^{(n)}~\text{sets a record}) = \P(K_n \geq 0) \sim \frac{1}{n}\frac{(\ln n)^{d - 1}}{(d - 1)!}
\end{equation}
as $n \to \infty$; see, for example, \cite[(4.5)]{Fillboundary(2020)}.
\smallskip 

(b)~Recall from \cite[proof of Theorem~1.4]{Fillboundary(2020)} that, conditionally given $K_n \geq 0$, the 
joint density of
\[
G_n = \| \XX^{(n)} \| - \L n, \quad \UU_n = \| \XX^{(n)} \|^{-1} \XX^{(n)}
\]
with respect to the product of Lebesgue measure on~$\bbR$ and uniform distribution on the probability 
simplex $\cS_{d - 1}$ is
\begin{equation}
\label{jtdens}
(g, \uu) \mapsto p_n^{-1} n^{-1} \frac{(g + \L n)^{d - 1}}{(d - 1)!} e^{ - g} (1 - n^{-1} e^{ - g})^{n - 1} {\bf 1}(g > - \L n)
\end{equation}
and, as $n \to \infty$, converges pointwise to
\[
(g, \uu) \mapsto e^{ - g} \exp\left( - e^{ - g} \right),
\]
the density (with respect to the same product measure) of the standard Gumbel probability measure and uniform measure on $\cS_{d - 1}$.  Thus by Scheff\'{e}'s theorem (\eg, \cite[Theorem 16.12]{Billingsley(2012)}), there is total variation convergence of 
$\cL(G_n, \UU_n)$ to $\cL(G, \UU)$ in obvious notation.
\smallskip

(c)~We also claim that
\begin{equation}
\label{intlim}
\E e^{ - G_n} \to \int_{- \infty}^{\infty} e^{-2 g} \exp(- e^{- g}) \dd g = 1
\end{equation}
as $n \to \infty$.  
To see this, first observe that as $n \to \infty$ we have [using~\eqref{jtdens} and~\eqref{pn}] that 
\begin{align*}
\E e^{ - G_n}
&= \int_{- \infty}^{\infty}\!{\bf 1}(g > - \L n)\,p_n^{-1}\,\frac{(g + \L n)^{d - 1}}{(d - 1)!}\,\frac{e^{-2 g}}{n}\,
\left( 1 - n^{-1} e^{ - g} \right)^{n - 1} \dd g \\
&\sim \int_{- \infty}^{\infty}\!{\bf 1}(g > - \L n)\,\left(1 + \frac{g}{\L n} \right)^{d - 1}\,e^{-2 g}\,
\left( 1 - n^{-1} e^{ - g} \right)^{n - 1} \dd g. 
\end{align*}
For fixed $g \in \bbR$, as $n \to \infty$ the integrand converges to
\[
e^{ - 2 g} \exp\left( - e^{ - g} \right),
\]
so it suffices to invoke the dominated convergence theorem.  For $n \geq 3$ we can dominate the integrand by
\[
(1 + |g|)^{d - 1} e^{ - 2 z} \exp\left( - \frac{n - 1}{n} e^{ - g} \right) 
\leq (1 + |g|)^{d - 1} e^{ - 2 g} \exp\left( - \frac{2}{3} e^{ - g} \right), 
\]
and this last expression integrates: Indeed, as $g \to - \infty$, it is asymptotically equivalent to 
$|g|^{d - 1} e^{2 |g|} \exp\left( - \frac{2}{3} e^{|g|} \right)$, and as $g \to \infty$, it is asymptotically equivalent to 
$g^{d - 1} e^{- 2 g}$.
\end{remark}
\smallskip

Here now is a heuristic proof of \refT{T:main}.  
Recall that we use the shorthand notation $\L$ for natural logarithm.
From \refR{R:Fillboundary}(a), the conditional density of $G_n = X^{(n)}_+ - \L n$ given $K_n \geq 0$ converges pointwise to the standard Gumbel density as $n \to \infty$; further, the conditional distribution of $\XX^{(n)}$ given $K_n \geq 0$ and $G_n = g$ is uniform over all positive $d$-tuples summing to $\L n + g$.  Next, given $K_n \geq 0$ and $\XX^{(n)} = \ww \succ 0$ with $w_+ = \L n + g$ (call this condition~$C$), the random variables $\XX^{(1)} - \ww, \ldots, \XX^{(n - 1)} - \ww$ are \iid,\ each with (conditional) density proportional to $\zz \mapsto e^{-z_+}$ with respect to Lebesgue measure on 
\[
S_n := \{ \zz \in \bbR^d: \zz \succ - \ww\mbox{\ and\ }\zz \not\succ {\bf 0} \} 
= \{ \zz \in \bbR^d: \zz \succ - \ww \} - \bbR_+^d,
\]
a proper set difference; the proportionality constant is then the reciprocal of
$e^{w_+} - 1 = n e^g - 1$.  For $\gD > 0$ and~$g$ fixed real numbers, this implies that the conditional density in question evaluated at a point $\zz \not\succ 0$ at distance (say, $\ell_1$-distance) at most~$\gD$ from~${\bf 0}$ is asymptotically equivalent to $n^{-1} e^{- z_+ - g}$.
     
It follows that (still conditionally given~$C$) the set of points $\XX^{(i)} - \ww + \xx(g)$ 
with $i \in \{1, \ldots, n - 1\}$ that are within distance~$\gD$ of $\xx(g)$ is approximately distributed (when~$n$ is large) as a nonhomogeneous Poisson point process with intensity function as described in \refT{T:main}, restricted to the ball of radius~$\gD$ centered at~$\xx(g)$.  The maxima of this restricted Poisson process ought to be the same as the maxima of the unrestricted process when~$\gD$ is sufficiently large.  Letting $\gD \to \infty$, \refT{T:main} results.

\section{Proof of \refT{T:main}}
\label{S:formal}

In this section we give a complete formal proof of the main \refT{T:main}.  
Let $\mu_n := \cL(K_n \mid K_n \geq 0)$.  In \refSS{S:limit} we prove the existence of a weak limit~$\mu$ for 
$\mu_n$, and in \refSS{S:identify} we identify~$\mu$ as the law of~$\cK$ as described in \refT{T:main}.

\subsection{Existence of weak limit}
\label{S:limit}

In this subsection we prove the following proposition.

\begin{proposition}
\label{P:limit}
There exists a probability measure~$\mu$ on $\{1, 2, \dots\}$ such that $\mu_n = \cL(K_n \mid K_n \geq 0)$ converges weakly to~$\mu$.
\end{proposition}

Indeed, \refP{P:limit} is established by showing that the sequence $(\mu_n)$ is tight (\refL{L:tight}) and has a vague limit (\refL{L:vague}).

\begin{lemma}
\label{L:tight}
\ \\

\vspace{-.1in}
{\rm (i)}~We have $\E(K_n \mid K_n \geq 0) \to 1$ as $n \to \infty$.

{\rm (ii)}~The sequence $(\mu_n)$ is tight. 
\end{lemma}

\begin{proof}
By Markov's inequality, boundedness of first moments is a sufficient condition for tightness.  Thus~(ii) follows from~(i).

We now prove~(i).  When $d = 1$ we have $K_n = 1$ deterministically if $K_n \geq 0$, so~(i) is obvious.  

We therefore assume henceforth that $d \geq 2$.
Recalling \refD{D:record}(d) and introducing the dimension~$d$ into the notation,
denote the number of records that have been broken through time~$n$ by $\beta_n(d) := R_n(d) - r_n(d)$. Then
\begin{align}
\E(K_n\,|\,K_n \geq 0)
&= \E[\beta_n(d) - \beta_{n - 1}(d)\,|\,\XX^{(n)}~\text{sets a record}] \nonumber\\
&= \dfrac{\E[\beta_n(d) - \beta_{n - 1}(d);\,\XX^{(n)}~\text{sets a record}]}{\P(\XX^{(n)}~\text{sets a record})}
\nonumber\\
&= \frac{\E[\beta_n(d) - \beta_{n - 1}(d)]}{\P_d(\XX^{(n)}~\text{sets a record})}. \label{eq:expect_K_n}
\end{align}
For the numerator of \eqref{eq:expect_K_n} we observe
\begin{align*}
\E[\beta_n(d) - \beta_{n - 1}(d)] 
&= \E[R_n(d) - R_{n - 1}(d)] - \E r_n(d) + \E r_{n - 1}(d) \\
&=\P_d(\XX^{(n)}~\text{sets a record}) - \E R_n(d - 1) + \E R_{n - 1}(d - 1) \\
&=\P_d(\XX^{(n)}~\text{sets a record}) - \P_{d - 1}(\XX^{(n)}~\text{sets a record}),
\end{align*}
where the second equality follows by standard consideration of concomitants: The two random variables $r_n(d)$  and 
$R_n(d - 1)$ have the same distribution, for any~$n$ and~$d$. Combining the last display 
with~\eqref{eq:expect_K_n}, one has
\begin{equation}
\E(K_n\,|\,K_n \geq 0) = 1 - \frac{\P_{d - 1}(\XX^{(n)}~\text{sets a record})}{\P_d(\XX^{(n)}~\text{sets a record})}. \label{eq:expec_K_n_simplified}
\end{equation}
Thus, by~\eqref{pn}, as $n \to \infty$ we have
\[
\E(K_n\,|\,K_n \geq 0) = 1 - (1 + o(1)) \frac{d - 1}{\ln n} \to 1, 
\]
as claimed.
\end{proof}

\begin{lemma}
\label{L:vague}
The sequence of probability measures $\mu_n = \cL(K_n \mid K_n \geq 0)$ converges vaguely.
\end{lemma}

\begin{proof}
It is sufficient to show that there exists $(\gk_k)_{k = 1, 2, \dots}$ such that for each~$k$ we have
\begin{equation}
\label{gk limit}
\P(K_n \geq k \mid K_n \geq 0) \to \gk_k\mbox{\ as $n \to \infty$}.
\end{equation}
This is proved by means of the next three lemmas and the arguments interspersed among them.
\end{proof}

\begin{lemma}
\label{L:knd+}
Let $\gD > 0$.  Let $K_n(\gD+) = -1/2$ if $\XX^{(n)}$ does not set a record, and otherwise let $K_n(\gD+)$ denote the number of remaining records killed by $\XX^{(n)}$ that are at $\ell^1$-distance $> \gD$ from $\XX^{(n)}$.  Then
\begin{align}
\limsup_{n \to \infty} \P(K_n(\gD+) \geq 1 \mid K_n \geq 0) 
&\leq \lim_{n \to \infty} \E (K_n(\gD+) \mid K_n \geq 0) \\
\label{gamma asy}
&= \P(\mbox{\rm Gamma$(d)$} > \gD) \sim e^{-\gD} \frac{\gD^{d - 1}}{(d - 1)!} \to 0
\end{align}
where {\rm Gamma$(d)$} is distributed {\rm Gamma}$(d, 1)$ and the asymptotics in~\eqref{gamma asy} are as 
$\gD \to \infty$.
\end{lemma}

\begin{proof}
Recalling that we use $\| \cdot \|$ to denote $\ell^{1}$-norm,
\begin{align}
\lefteqn{\hspace{-.15in} \P(K_n \geq 0,\,K_n(\gD+) \geq 1)} \nonumber \\ 
&\leq \E(K_n(\gD+); K_n \geq 0) \nonumber \\
&= (n - 1) \P(K_{n - 1} \geq 0,\,\XX^{(n)} \succ \XX^{(n - 1)},\,\| \XX^{(n)} - \XX^{(n - 1)} \| > \gD).
\label{knd+num}
\end{align}
But
\begin{align*}
\lefteqn{\hspace{-.4in}\P(K_{n - 1} \geq 0,\,\XX^{(n)} \succ \XX^{(n - 1)},\,\| \XX^{(n)} - \XX^{(n - 1)} \| > \gD)} \\
&= \int_{\zz, \xx:{\bf 0} \prec \zz \prec \xx,\,\| \xx - \zz \| > \gD}\!
\P\left( \XX^{(n)} \in \ddx \xx;\,\XX^{(n - 1)} \in \ddx \zz; \right. \\
&{} \hspace{1.8in} \left. \XX^{(i)} \not\succ \zz\mbox{\ for $i = 1, \dots, n - 2$} \right) \\
&= \int_{\zz, \xx:{\bf 0} \prec \zz \prec \xx,\,\| \xx - \zz \| > \gD}\!
e^{ - \|\xx\|}\,e^{ - \|\zz\|}\,\left(1 - e^{ - \|\zz\|} \right)^{n - 2} \dd \xx \dd \zz \\
&= \int_{\zz \succ {\bf 0}} \int_{\gdgd \succ {\bf 0}:\|\gdgd\| > \gD}\!
e^{ - (\|\zz\| + \|\gdgd\|)}\,e^{ - \|\zz\|}\,\left(1 - e^{ - \|\zz\|} \right)^{n - 2} \dd \gdgd \dd \zz \\
&= \left[ \int_{\gD}^{\infty}\!\frac{y^{d - 1}}{(d - 1)!} e^{- y} \dd y \right] 
\times \int_{\zz \succ {\bf 0}}\,e^{ - 2 \|\zz\|}\,\left(1 - e^{ - \|\zz\|} \right)^{n - 2} \dd \zz.
\end{align*}
The first factor 
equals $\P(\mbox{Gamma}(d) > \gD)$.  Utilizing all three parts of \refR{R:Fillboundary}, (i)~the second factor, when divided by $p_{n - 1}$, equals 
\[
\E \left( e^{ - \| \XX^{(n - 1)} \|} \mid \mbox{$\XX^{(n - 1)}$ sets a record} \right)
= (n - 1)^{-1} \E e^{ - G_{n - 1}};
\]
and thus (ii) $\E (K_n(\gD+) \mid K_n \geq 0)$ equals
\begin{align*}
\P(\mbox{Gamma}(d) > \gD) \times \frac{p_{n - 1}}{p_n} \E e^{ - G_{n - 1}}
&\sim \P(\mbox{Gamma}(d) > \gD) \times \E e^{ - G_{n - 1}} \\ 
&\to \P(\mbox{Gamma}(d) > \gD),
\end{align*}
as $n \to \infty$.  The tail-asymptotics for $\gG(d)$ appearing in~\eqref{gamma asy} are standard.  This completes the proof of the lemma.
\end{proof}

Having suitably controlled $K_n(\gD+)$, we turn our attention to $K_n(\gD-)$, defined as follows for given $\gD > 0$.
Let $K_n(\gD-) = -1/2$ if $\XX^{(n)}$ does not set a record, and otherwise let $K_n(\gD-)$ denote the number of remaining records killed by $\XX^{(n)}$ that are at $\ell^1$-distance $\leq \gD$ from $\XX^{(n)}$; thus 
\begin{equation}
\label{Kn decomp}
K_n = K_n(\gD-) + K_n(\gD+).
\end{equation}

\ignore{
Recall from \cite[proof of Theorem~1.4]{Fillboundary(2020)} that, conditionally given $K_n \geq 0$, the 
joint density of
\[
G_n = \| \XX^{(n)} \| - \L n, \quad \UU_n = \| \XX^{(n)} \|^{-1} \XX^{(n)}
\]
with respect to the product of Lebesgue measure on~$\bbR$ and uniform distribution on the probability 
simplex $\cS_{d - 1}$ is
\[
(g, \uu) \mapsto p_n^{-1} n^{-1} \frac{(g + \L n)^{d - 1}}{(d - 1)!} e^{ - g} (1 - n^{-1} e^{ - g})^{n - 1} {\bf 1}(g > - \L n)
\] 
and, as $n \to \infty$, converges pointwise to
\[
(g, \uu) \mapsto e^{ - g} \exp\left( - e^{ - g} \right),
\]
the density (with respect to the same product measure) of the standard Gumbel probability measure and uniform measure on $\cS_{d - 1}$.  Thus by Scheff\'{e}'s theorem (\eg, \cite[Theorem 16.12]{Billingsley(2012)}), there is total variation convergence of 
$\cL(G_n, \UU_n)$ to $\cL(G, \UU)$ in obvious notation.

It therefore follows that 
}
It follows using \refR{R:Fillboundary}(b) that 
if we can prove, for each $k \in \{ 1, 2, \ldots \}$, that
\begin{equation}
\label{gu int limit}
\int_{g, \uu} \P(G \in \ddx g,\,\UU \in \ddx u)\,\P(K_n(\gD-) \geq k \mid K_n \geq 0,\,G_n = g, \UU_n = \uu)
\end{equation}
has a limit, call it $\gk_k(\gD-)$, as $n \to \infty$, then, also for each~$k$, we have
\begin{equation}
\label{close limit}
\P(K_n(\gD-) \geq k \mid K_n \geq 0) \to \gk_k(\gD-)\mbox{\ as $n \to \infty$}.
\end{equation}

To prove that~\eqref{gu int limit} has a limit, it suffices by the dominated convergence to prove the following lemma, in which case we can then take
\begin{equation}
\label{gk close}
\gk_k(\gD-) := \int_g \P(G \in \ddx g)\,\gk_k(\gD-; \infty, g).
\end{equation}

\begin{lemma}
\label{L:closelimit}
Let $k \in \{1, 2, \dots\}$, $g \in \bbR$, and $\uu \in \cS_{d - 1}$ with $\uu \succ {\bf 0}$.  Then
\[
\gk_k(\gD-; n, g, \uu) := \P(K_n(\gD-) \geq k \mid K_n \geq 0,\,G_n = g,\,\UU_n = \uu)
\]
has a limit $\gk_k(\gD-; \infty, g)$ as $n \to \infty$, which doesn't depend on~$\uu$.
\end{lemma}

\begin{proof}
We use the method of moments [and so, by the way, $\gk_k(\gD-; \infty, g) \downarrow 0$ as $k \uparrow \infty$; that is, the conditional distributions in question have a \emph{weak} limit].  
Let $\xx = (\L n + g) \uu$ (which implies $\| \xx \| = \L n + g$), and let $R \equiv R(\gD; n - 1, \xx)$ denote the set of remaining record-values $\XX^{(i)}$ at time $n - 1$ that are killed by $\XX^{(n)} = \xx$ at time~$n$ and satisfy 
$\| \xx - \XX^{(i)} \| \leq \gD$.  By writing $K_n(\gD-)$ as a sum of indicators, for integer $r \geq 1$ we calculate
\begin{align}
\label{moment r}
\lefteqn{\E[(K_n(\gD-))^r \mid K_n \geq 0,\,G_n = g,\,\UU_n = \uu]} \\
\label{triple}
&= \sum_{m = 1}^r \sum \sum {r \choose {j_1, \dots, j_m}}
\P(\{ \XX^{(i_1)}, \ldots, \XX^{(i_m)} \} \subseteq R \mid \mbox{$\XX^{(n)} = \xx$ sets a record}) \\
\label{brm}
&= \sum_{m = 1}^r {{n - 1} \choose m} m! {r \brace m}
\P(\{ \XX^{(1)}, \ldots, \XX^{(m)} \} \subseteq R \mid \mbox{$\XX^{(n)} = \xx$ sets a record}),
\end{align}
where, in \eqref{triple}, the second of the three sums is over $i_1, \dots, i_m$ satisfying
$1 \leq i_1 < \cdots < i_m \leq n - 1$ and the third sum is over $j_1, \dots, j_m \geq 1$ satisfying $j_+ = r$;
and, in \eqref{brm}, ${r \brace m}$ is a Stirling number of the second kind and is the number of ways to partition an $r$-set into~$m$ nonempty sets.

Now, writing $\XX = \XX^{(1)}$, the conditional probability in~\eqref{brm} equals
\[
\frac{1}{[\P(\XX \not\succ \xx)]^{n - 1}} 
\int\!
\left[ \prod_{i = 1}^m \left( e^{ - \| \xx^{(i)} \|} \dd \xx^{(i)} \right) \right]
\left[ \P\left( \bigcap_{i = 1}^m \left\{ \XX \not\succ \xx^{(i)} \right\} \right) \right]^{n - m},
\]
where the integral is over incomparable vectors $\xx^{(1)}, \dots, \xx^{(m)}$ such that, for $i = 1, \dots, m$, we have ${\bf 0} \prec \xx^{(i)} \prec \xx$ and $\| \xx - \xx^{(i)} \| \leq \gD$.
For the denominator here we calculate
\[
[\P(\XX \not\succ \xx)]^{n - 1} = \left( 1 - e^{ - \| \xx \|} \right)^{n - 1} = \left( 1 - n^{-1} e^{-g} \right)^{n - 1} 
\to \exp\left( - e^{ - g} \right).
\]
Changing variables from $\xx^{(i)}$ to $\gdgd^{(i)} = \xx - \xx^{(i)}$, the numerator (\ie,\ integral) can be written as
\begin{align}
\label{firstint}
\lefteqn{\int\!e^{ - m \| \xx \|}
\left[ \prod_{i = 1}^m \left( e^{ \| \gdgd^{(i)} \|} \dd \gdgd^{(i)} \right) \right]
\left[ \P\left( \bigcap_{i = 1}^m \left\{ \XX \not\succ \xx - \gdgd^{(i)} \right\} \right) \right]^{n - m}} \\
\label{secondint}
&= n^{-m} e^{ - m g}
\int\! A
\left[ \prod_{i = 1}^m \left( e^{ \| \gdgd^{(i)} \|} \dd \gdgd^{(i)} \right) \right]
\left[ \P\left( \bigcap_{i = 1}^m \left\{ \XX \not\succ \xx - \gdgd^{(i)} \right\} \right) \right]^{n - m}.
\end{align}
Both integrals are over vectors $\gdgd^{(1)}, \dots, \gdgd^{(m)}$ satisfying the restrictions that 
\begin{align}
\label{restrict}
\lefteqn{\mbox{$\gdgd^{(1)}, \dots, \gdgd^{(m)}$ are incomparable, and,}} \\ 
&{} \qquad \quad \mbox{for $i = 1, \dots, m$, we have ${\bf 0} \prec \gdgd^{(i)}$ and $\| \gdgd^{(i)} \| \leq \gD$}; \nonumber 
\end{align}
for the integral in~\eqref{firstint} we have the additional restrictions that $\gdgd^{(i)} \prec \xx$, and, correspondingly, in the second integral we use the shorthand 
\[
 A = {\bf 1}(\gdgd^{(i)} \prec \xx\mbox{\ for $i = 1, \dots, m$}).
 \]
Observe that the integrands in~\eqref{secondint} can be dominated by
\begin{equation}
\label{dom}
\prod_{i = 1}^m e^{ \| \gdgd^{(i)} \|},
\end{equation}
which is integrable [over the specified range~\eqref{restrict} of $\gdgd^{(1)}, \dots, \gdgd^{(m)}$] because, dropping the restriction of incomparability in the next integral to appear, the integral of~\eqref{dom} can be bounded by 
\begin{align}
\int\! \prod_{i = 1}^m \left( e^{ \| \gdgd^{(i)} \|} \dd \gdgd^{(i)} \right)
&= \left[ \int_{\gdgd \succ {\bf 0}:\| \gdgd \| \leq \gD}\!e^{\| \gdgd \|} \dd \gdgd \right]^m \nonumber \\
&= \left[ \int_0^{\gD}\!\frac{y^{d - 1}}{(d - 1)!}\,e^y \dd y \right]^m \nonumber \\
\label{gdintbound}
&\leq \left[ \int_0^{\gD}\!e^{2 y} \dd y \right]^m \nonumber \\ 
&\leq \left( \frac{1}{2} e^{2 \gD} \right)^m 
= 2^{ - m} e^{2 \gD m} 
< \infty. 
\end{align}

So we may apply the dominated convergence theorem to the integral appearing in~\eqref{secondint}.  The assumption $\uu \succ {\bf 0}$ of strict positivity is crucial, since it implies that $A \to 1$ as $n \to \infty$.  If we can show, for fixed~$m$ and~$g$ and~$\uu$ and $\gdgd^{(1)}, \ldots \gdgd^{(m)}$ that 
\begin{align}
q_n &\equiv q_n(m, g, \uu, \gdgd^{(1)}, \dots, \gdgd^{(m)}) 
:= \left[ \P\left( \bigcap_{i = 1}^m \left\{ \XX \not\succ \xx - \gdgd^{(i)} \right\} \right) \right]^{n - m} \nonumber \\
\label{q} 
&{} \quad \mbox{has a limit $q_{\infty}(m, g, \uu, \gdgd^{(1)}, \dots, \gdgd^{(m)})$ as $n \to \infty$},
\end{align}
then~\eqref{moment r} has the following limit as $n \to \infty$:
\begin{align}
\lefteqn{\hspace{-.7in}\exp\!\left( e^{ - g} \right) \sum_{m = 1}^r {r \brace m} e^{ - m g}} \nonumber \\
&\times \int\! \left[ \prod_{i = 1}^m \left( e^{ \| \gdgd^{(i)} \|} \dd \gdgd^{(i)} \right) \right] 
q_{\infty}(m, g, \uu, \gdgd^{(1)}, \dots, \gdgd^{(m)}),
\label{lim moment r}
\end{align}
where the integral is over $\gdgd^{(1)}, \ldots \gdgd^{(m)}$ satisfying~\eqref{restrict}.  Further, using~\eqref{gdintbound} this limit is bounded by
\begin{align*}
\exp\!\left( e^{ - g} \right) \sum_{m = 1}^r {r \brace m} e^{ - m g} 2^{ - m} e^{2 \gD m}
&\leq \exp\!\left( e^{ - g} \right) \sum_{m = 1}^r \frac{m^r}{m!} e^{ - m g} 2^{ - m} e^{2 \gD m} \\
&\leq \frac{1}{2} \exp\!\left( e^{ - g} \right) r^{r + 1} e^{r |g|} e^{2 \gD r}.
\end{align*}
Given the rate of growth of this bound as a function of~$r$, the method of moments applies:\ The conditional distribution of $K_n(\gD-)$ given $K_n \geq 0$, $G_n = g$, and $\UU_n = \uu$ converges weakly to the unique probability measure on the nonnegative integers whose $r$th moment is given by~\eqref{lim moment r} for $r = 1, 2, \dots$.

It remains to prove~\eqref{q}.  Indeed, defining $\gdgd^{(i_1, \ldots, i_{\ell})}$ to be the coordinate-wise minimum $\wedge_{h = 1}^{\ell} \gdgd^{(i_h)}$ and applying inclusion--exclusion at the second equality, 
\begin{align}
q_n
&= \left[ 1 - \P\left( \bigcup_{i = 1}^m \left\{ \XX \succ \xx - \gdgd^{(i)} \right\} \right) \right]^{n - m} \nonumber \\
&= \left[ 1 - \sum_{\ell = 1}^m (-1)^{\ell - 1} \sum_{1 \leq i_1 < i_2 < \cdots < i_{\ell} \leq m}  
\P\left( \bigcap_{h = 1}^{\ell} \left\{ \XX \succ \xx - \gdgd^{(i_h)} \right\} \right) \right]^{n - m} \nonumber \\
&= \left[ 1 - \sum_{\ell = 1}^m (-1)^{\ell - 1} \sum_{1 \leq i_1 < i_2 < \cdots < i_{\ell} \leq m}  
\P\left( \XX \succ \xx - \gdgd^{(i_1, \ldots, i_{\ell}))} \right) \right]^{n - m} \nonumber \\
&= \left[ 1 - e^{ - \| \xx \|} \sum_{\ell = 1}^m (-1)^{\ell - 1} \sum_{1 \leq i_1 < i_2 < \cdots < i_{\ell} \leq m}  
e^{\| \gdgd^{(i_1, \ldots, i_{\ell}))} \|} \right]^{n - m} \nonumber \\
&= \left[ 1 - n^{-1} e^{ - g} \sum_{\ell = 1}^m (-1)^{\ell - 1} \sum_{1 \leq i_1 < i_2 < \cdots < i_{\ell} \leq m}  
e^{\| \gdgd^{(i_1, \ldots, i_{\ell}))} \|} \right]^{n - m} \nonumber \\
&\tend \exp\left[ - e^{ - g} \sum_{\ell = 1}^m (-1)^{\ell - 1} \sum_{1 \leq i_1 < i_2 < \cdots < i_{\ell} \leq m}  
e^{\| \gdgd^{(i_1, \ldots, i_{\ell}))} \|} \right]\mbox{\ as $n \to \infty$} \nonumber \\
&=: q_{\infty}(m, g, \gdgd^{(1)}, \dots, \gdgd^{(m)}),
\label{qinfinity} 
\end{align}
as desired.
\end{proof}

Since the expression $\gk_k(\gD; n, g, \uu)$ studied in \refL{L:closelimit} is clearly nondecreasing in~$\gD$, the same is true for its limit $\gk_k(\gD; \infty, g)$---and therefore, by~\eqref{gk close}, for the limit $\gk_k(\gD-)$ appearing in~\eqref{close limit}.

\begin{lemma}
\label{L:explicit}
The convergence
\[
\P(K_n \geq k \mid K_n \geq 0) \to \gk_k\mbox{\ as $n \to \infty$}
\]
claimed at~\eqref{gk limit} holds with
\begin{equation}
\label{gkk explicit}
\gk_k := \lim_{\gD \to \infty} \gk_k(\gD-),
\end{equation}
where $\gk_k(\gD-)$ is defined at~\eqref{gk close} using \refL{L:closelimit}.
\end{lemma}

\begin{proof}
This follows simply from~\eqref{Kn decomp}, \refL{L:knd+}, and~\eqref{close limit}.  Indeed, the conditional probability in question is, for each~$\gD$, at least as large as the one in~\eqref{close limit} and so, by~\eqref{close limit}, has a $\liminf$ at least $\gk_k(\gD-)$.  Letting $\gD \to \infty$, we find
\[
\liminf_{n \to \infty} \P(K_n \geq k \mid K_n \geq 0) \geq \gk_k.
\]
Conversely, the conditional probability in question is, by finite subadditivity, at most the sum of the one in~\eqref{close limit} and the one treated in \refL{L:knd+}.  Thus, for each~$\gD$, we have
\[
\limsup_{n \to \infty} \P(K_n \geq k \mid K_n \geq 0) \leq \gk_k(\gD-) + \P(\mbox{Gamma}(d) > \gD)  
\] 
by~\eqref{close limit} and \refL{L:knd+}.  Letting $\gD \to \infty$, we find from~\eqref{gkk explicit} and~\eqref{gamma asy} that
\[
\limsup_{n \to \infty} \P(K_n \geq k \mid K_n \geq 0) \leq \gk_k.
\] 
This completes the proof. 
\end{proof}

\subsection{Identification of the weak limit~$\mu$ as $\cL(\cK)$}
\label{S:identify}

In this subsection we prove the following proposition.  Recall that the existence of a weak limit~$\mu$ for the probability measures $\mu_n = \cL(K_n \mid K_n \geq 0)$ is established in \refP{P:limit} and identified to some extent in \refL{L:explicit}.

\begin{proposition}
\label{P:identify}
The weak limit~$\mu$ of the measures $\mu_n = \cL(K_n \mid K_n \geq 0)$ is the distribution $\cL(\cK)$ described in \refT{T:main}.
\end{proposition}

Our approach to proving this proposition is to define $\cK(\gD-)$ and $\cK(\gD+)$ in relation to~$\cK$ in the same way that we defined $K_n(\gD-)$ and $K_n(\gD+)$ in relation to $K_n$, to prove an analogue (namely, \refL{L:kd+}) of \refL{L:knd+}, and to use again the method of moments to establish (see \refL{L:close}) that the limit $\gk_k(\gD-; \infty, g)$ in \refL{L:closelimit} equals $\P(\cK_g(\gD-) \geq k)$ (in notation that should be obvious and which we shall at any rate make explicit in the statement of \refL{L:close}).

After stating and proving Lemmas~\ref{L:kd+}--\ref{L:close}, we will give the (then) easy proof of \refP{P:identify}. 

\begin{lemma}
\label{L:kd+}
Let $\gD > 0$.  For $g \in \bbR$, let
\begin{align*}
\cK_g(\gD+) 
&= \mbox{\rm number of maxima dominated by $\xx \equiv \xx(g) := (g / d, \ldots, g / d)$} \\ 
&{} \qquad \mbox{\rm and at $\ell^1$-distance $> \gD$ from~$\xx$}
\end{align*}
in the nonhomogeneous Poisson point process described in \refT{T:main}, with intensity function
\[
{\bf 1}(\zz \not\succ \xx) e^{- z_+}, \quad \zz \in \bbR^d.
\]
Define $\cL(\cK(\gD+))$ to be the mixture $\E \cL(\cK_G(\gD+))$, where~$G$ is distributed standard Gumbel.
Then
\[
\P(\cK(\gD+) \geq 1) \leq \E \cK(\gD+) = \P(\mbox{\rm Gamma}(d) > \gD) \sim e^{-\gD} \frac{\gD^{d - 1}}{(d - 1)!} \to 0
\]
where $\mbox{\rm Gamma}(d)$ is distributed {\rm Gamma}$(d, 1)$ and the asymptotics are as $\gD \to \infty$.
\end{lemma}

\begin{proof}
Let
\[
\cP := \{ \yy \in \bbR^d:\yy \prec \xx,\,\| \xx - \yy \| > \gD \}.  
\]
Denote the Poisson process by $N_g$.  Apply the so-called Mecke equation by setting
\[
f(\zz, \eta) = {\bf 1}(\mbox{$\zz$ is a maximum of~$\eta$ dominated by~$\xx$ and at $\ell^1$-distance from~$\xx$})
\]
and taking $\eta = N_g$ in \cite[Theorem~4.1]{Last(2018)} to find
\begin{align*}
\lefteqn{\E \cK_g(\gD+)} \\
&= \int_{\yy \in \cP} e^{ - y_+}\,
\exp\!\left( - \int_{\zz:\zz \succ \yy,\,\zz \not\succ \xx}\!e^{ - z_+} \dd \zz \right) \dd \yy \\
&= \int_{\yy \in \cP} e^{ - y_+}\,
\exp\!\left[ - \left( \int_{\zz:\zz \succ \yy}\!e^{ - z_+} \dd \zz - \int_{\zz:\zz \succ \xx}\!e^{ - z_+} \dd \zz \right) \right] \dd \yy \\
&= \int_{\yy \in \cP} e^{ - y_+}\,\exp\!\left[ - \left( e^{ - y_+} - e^{ - x_+} \right) \right] \dd \yy \\
&= \exp\!\left( e^{ - g} \right) \int_{\yy \in \cP}\!e^{ - y_+}\,\exp\!\left( - e^{ - y_+} \right) \dd \yy \\
&= \exp\!\left( e^{ - g} \right) \int_{\gdgd \succ {\bf 0}:\| \gdgd \| > \gD}\!e^{ - \left( x_+ - \| \gdgd \| \right)}\,
\exp\!\left[ - e^{ - \left( x_+ - \| \gdgd \| \right)} \right] \dd \gdgd \\
&= \exp\!\left( e^{ - g} \right) \int_{\gdgd \succ {\bf 0}:\| \gdgd \| > \gD}\!e^{ - \left( g - \| \gdgd \| \right)}\,
\exp\!\left[ - e^{ - \left( g - \| \gdgd \| \right)} \right] \dd \gdgd \\
&= \exp\!\left( e^{ - g} \right) \int_{\gD}^{\infty}\!\frac{\eta^{d - 1}}{(d - 1)!}\,e^{ - (g - \eta)}\,
\exp\!\left[ - e^{ - (g - \eta)} \right] \dd \eta.
\end{align*}
Multiply by $\P(G \in \ddx g) = \exp\!\left( - e^{ - g} \right) e^{ - g} \dd g$ and integrate over $g \in \bbR$ to get
\begin{align*}
\P(\cK(\gD+) \geq 1) 
&\leq \E \cK(\gD+) \\
&= \int_{\gD}^{\infty}\!\frac{\eta^{d - 1}}{(d - 1)!}\,e^{\eta} \int_{- \infty}^{\infty}\!e^{ - 2 g}\,
\exp\!\left( - e^{\eta} e^{ - g} \right) \dd g \dd \eta \\
&= \int_{\gD}^{\infty}\!\frac{\eta^{d - 1}}{(d - 1)!}\,e^{ - \eta} \dd \eta 
= \P(\mbox{Gamma}(d) > \gD),
\end{align*}
as claimed.
\end{proof}

\begin{lemma}
\label{L:close}
Let $\gD > 0$.  For $g \in \bbR$, let
\begin{align*}
\cK_g(\gD-) 
&= \mbox{\rm number of maxima dominated by $\xx \equiv \xx(g) := (g / d, \ldots, g / d)$} \\ 
&{} \qquad \mbox{\rm and at $\ell^1$-distance $\leq \gD$ from~$\xx$}
\end{align*}
in the nonhomogeneous Poisson point process described in \refT{T:main}, with intensity function
\[
{\bf 1}(\zz \not\succ \xx) e^{- z_+}, \quad \zz \in \bbR^d.
\]
Then for every $k \in \{1, 2, \dots\}$ we have
\[
\P(\cK_g(\gD-) \geq k) = \gk_k(\gD-; \infty, g)
\]
with $\gk_k(\gD-; \infty, g)$ as in \refL{L:closelimit}.
\end{lemma}

\begin{proof}
We need only show that, for each $r = 1, 2, \dots$, the random variable $\cK_g(\gD-)$ has $r^{\rm \scriptsize th}$ moment equal to~\eqref{lim moment r}, where $q_{\infty}(m, g, \gdgd^{(1)}, \dots, \gdgd^{(m)})$ is defined at~\eqref{qinfinity}. 

For this, we proceed in much the same way as for the proof of \refL{L:kd+}.  In this case, let
\[
\cP := \{ \yy \in \bbR^d:\yy \prec \xx,\,\| \xx - \yy \| \leq \gD \}. 
\]

Applying the multivariate Mecke equation \cite[Theorem~4.4]{Last(2018)}, we find
\begin{align*}
\lefteqn{\E\left[ \left( \cK_g(\gD-) \right)^r \right]} \\
&= \sum_{m = 1}^r {r \brace m} 
\int\!\left[ \prod_{h = 1}^m \left( e^{ - x^{(h)}_+} \dd \xx^{(h)} \right) \right]
\exp\!\left( - \int_{\substack{\zz:\zz \succ \xx^{(h)}\mbox{\scriptsize\ for some~$h$}, \\ 
\hspace{-.65in}\zz \not\succ \xx}}
e^{ - z_+} \dd \zz \right) \\
&= \exp\!\left( e^{- g} \right) \sum_{m = 1}^r {r \brace m}
\int\!\left[ \prod_{h = 1}^m \left( e^{ - x^{(h)}_+} \dd \xx^{(h)} \right) \right]
\exp\!\left( - \int_{\zz:\zz \succ \xx^{(h)}\mbox{\scriptsize\ for some~$h$}} e^{ - z_+} \dd \zz \right),
\end{align*}
where the unlabeled integrals are over incomparable vectors $\xx^{(1)}, \dots, \xx^{(m)}$ belonging to~$\cP$.

Changing variables from $\xx^{(h)}$ to $\gdgd^{(h)} = \xx - \xx^{(h)}$, we find
\begin{align}
\lefteqn{\E\left[ \left( \cK_g(\gD-) \right)^r \right]} \nonumber \\
&= \exp\!\left( e^{- g} \right) \sum_{m = 1}^r {r \brace m} e^{ - m g} \nonumber \\ 
&{} \qquad \times \int\!\left[ \prod_{i = 1}^m \left( e^{\| \gdgd^{(i)} \|} \dd \gdgd^{(i)} \right) \right]
\exp\!\left( - \int_{\zz:\zz \succ \xx - \gdgd^{(i)}\mbox{\scriptsize\ for some~$i$}} e^{ - z_+} \dd \zz \right),
\label{EKgD-r}
\end{align}
where now the unlabeled integral is over vectors $\gdgd^{(1)}, \dots, \gdgd^{(m)}$ satisfying the restrictions~\eqref{restrict}.
By using inclusion--exclusion, the $\zz$-integral equals [writing $\gdgd^{(i_1, \dots, i_{\ell})}$ for the 
coordinate-wise minimum $\wedge_{h = 1}^{\ell} \gdgd^{(i_h)}$ as in the proof of \refL{L:closelimit}]
\begin{align*}
\lefteqn{\sum_{\ell = 1}^m (-1)^{\ell - 1} \sum_{1 \leq i_1 < i_2 < \dots < i_{\ell} \leq m} 
\int_{\zz:\zz \succ \xx - \gdgd^{(i_j)}\mbox{\scriptsize\ for all $1 \leq j \leq \ell$}}
e^{ - z_+} \dd \zz} \\
&= \sum_{\ell = 1}^m (-1)^{\ell - 1} \sum_{1 \leq i_1 < i_2 < \dots < i_{\ell} \leq m} 
\int_{\zz:\zz \succ \xx - \gdgd^{(i_1, \dots, i_{\ell})}} e^{ - z_+} \dd \zz \\
&= e^{ - g} \sum_{\ell = 1}^m (-1)^{\ell - 1} \sum_{1 \leq i_1 < i_2 < \dots < i_{\ell} \leq m}
e^{\| \gdgd^{(i_1, \dots, i_l)} \|}.
\end{align*}

Thus [confer~\eqref{qinfinity}]
\begin{align}
\lefteqn{\E\left[ \left( \cK_g(\gD-) \right)^r \right]} \nonumber \\
&= \exp\!\left( e^{- g} \right) \sum_{m = 1}^r {r \brace m} e^{ - m g} 
\label{rth moment} \\
&{} \times
\int\!\left[ \prod_{i = 1}^m \left( e^{\| \gdgd^{(i)} \|} \dd \gdgd^{(i)} \right) \right]
\exp\!\left( - e^{ - g} \sum_{\ell = 1}^m (-1)^{\ell - 1} \sum_{1 \leq i_1 < i_2 < \dots < i_{\ell} \leq m}
e^{\| \gdgd^{(i_1, \dots, i_l)} \|} \right) \nonumber \\
&= \exp\!\left( e^{- g} \right) \sum_{m = 1}^r {r \brace m} e^{ - m g}
\int\!\left[ \prod_{i = 1}^m \left( e^{\| \gdgd^{(i)} \|} \dd \gdgd^{(i)} \right) \right]
q_{\infty}(m, g, \gdgd^{(1)}, \dots, \gdgd^{(m)}) \nonumber \\
&= \eqref{lim moment r}, \nonumber
\end{align}
as claimed.
\end{proof}

\begin{proof}[Proof of \refP{P:identify}]
Clearly,
\begin{equation}
\label{1}
\P(\cK(\gD-) \geq k) \leq \P(\cK \geq k) \leq \P(\cK(\gD-) \geq k) + \P(\cK(\gD+) \geq 1).
\end{equation}
But
\begin{align}
\P(\cK(\gD-) \geq k) 
&= \int_{- \infty}^{\infty} \P(G \in \dd g) \P(\cK_g(\gD-) \geq k) \nonumber \\
&= \int_{- \infty}^{\infty} \P(G \in \dd g)\,\gk_k(\gD-; \infty, g) \label{2} \\
&= \gk_k(\gD-), \label{3}
\end{align}
using \refL{L:close} at~\eqref{2} and~\eqref{gk close} at~\eqref{3}.

Therefore, passing to the limit in~\eqref{1} using~\eqref{gkk explicit} and recalling \refL{L:explicit}, we find
\[
\P(\cK \geq k) = \gk_k = \lim_{n \to \infty} \P(K_n \geq k \mid K_n \geq 0) = \mu(\{k, k + 1, \dots\}),
\]
as claimed.
\end{proof}

\section{Tail probabilities for~$\cK$}
\label{S:tail}
In \refSS{S:moment bounds} we show that all the moments of $\cL(\cK)$ are finite and give a simple upper bound on each, and in \refSS{S:tail asymptotics} we study right-tail (logarithmic) asymptotics for $\cL(\cK)$. 

\subsection{Bounds on the moments of $\cL(\cK)$}
\label{S:moment bounds}

Here is the main result of this subsection.

\begin{proposition}
\label{P:moment bound}
Let~$\cK$ be as described in \refT{T:main}.  Then for $r = 1, 2, \dots$ we have
\[
\E \cK^r \leq a_d^r r^{(d + 2 -  \frac{1}{d - 1}) r} < \infty,
\]
with $a_d := 2^{1 / d} d^{(d + 1) / (d - 1)}$.
\end{proposition}

\begin{proof}
Recall from~\eqref{EKgD-r} that
\begin{align}
\lefteqn{\E\left[ \left( \cK_g(\gD-) \right)^r \right]} \nonumber \\
&= \exp\!\left( e^{- g} \right) \sum_{m = 1}^r {r \brace m} e^{ - m g} \nonumber \\ 
&{} \qquad \times \int\!\left[ \prod_{i = 1}^m \left( e^{\| \gdgd^{(i)} \|} \dd \gdgd^{(i)} \right) \right]
\exp\!\left( - \int_{\zz:\zz \succ \xx - \gdgd^{(i)}\mbox{\scriptsize\ for some~$i$}} e^{ - z_+} \dd \zz \right) 
\nonumber \\
&= \exp\!\left( e^{- g} \right) \sum_{m = 1}^r {r \brace m} e^{ - m g}
\label{EKgD-r again}\\ 
&{} \qquad \times \int\!\left[ \prod_{i = 1}^m \left( e^{\| \gdgd^{(i)} \|} \dd \gdgd^{(i)} \right) \right]
\exp\!\left( - e^{ - g} \int_{\gdgd:\gdgd \prec \gdgd^{(i)}\mbox{\scriptsize\ for some~$i$}} 
e^{ \gd_+} \dd \gdgd \right), \nonumber
\end{align}
where the unlabeled integral is over vectors $\gdgd^{(1)}, \ldots, \gdgd^{(m)}$ satisfying the restrictions~\eqref{restrict}.  

Next, multiply both sides of this equality by $\P(G \in \ddx g) = \exp\!\left( - e^{ - g} \right) e^{ - g} \dd g$ and integrate over $g \in \bbR$ to find
\[
\E\left[ \left( \cK(\gD-) \right)^r \right]
= \sum_{m = 1}^r {r \brace m} m!
\int\!\frac{\prod_{i = 1}^m \left( e^{\| \gdgd^{(i)} \|} \dd \gdgd^{(i)} \right)}
{\left[ \int_{\gdgd:\gdgd \prec \gdgd^{(i)}\mbox{\scriptsize\ for some~$i$}} 
e^{ \gd_+} \dd \gdgd \right]^{m + 1}}. 
\]

Changing variables by scaling,
\[
\E\left[ \left( \cK(\gD-) \right)^r \right]
= \sum_{m = 1}^r {r \brace m} m!\,m^{ d m}
\int\!\frac{\prod_{i = 1}^m \left( e^{m \| \gdgd^{(i)} \|} \dd \gdgd^{(i)} \right)}
{\left[ \int_{\gdgd:\gdgd \prec m \gdgd^{(i)}\mbox{\scriptsize\ for some~$i$}} 
e^{\gd_+} \dd \gdgd \right]^{m + 1}},
\] 
where now the unlabeled integral is over $\gdgd^{(1)}, \dots, \gdgd^{(m)}$ satisfying the restrictions
\begin{align*}
\lefteqn{\mbox{$\gdgd^{(1)}, \dots, \gdgd^{(m)}$ are incomparable, and,}} \\ 
&{} \qquad \quad \mbox{for $i = 1, \dots, m$, we have 
${\bf 0} \prec \gdgd^{(i)}$ and $\| \gdgd^{(i)} \| \leq \gD / m$}.
\end{align*}

Observe that for each $j = 1, \dots, m$ we have
\[
\int_{\gdgd:\gdgd \prec m \gdgd^{(i)}\mbox{\scriptsize\ for some~$i$}} e^{ \gd_+} \dd \gdgd
\geq \int_{\gdgd:\gdgd \prec m \gdgd^{(j)}} e^{\gd_+} \dd \gdgd = e^{m \| \gdgd^{(j)} \|}
\]
and hence, taking the geometric mean of the lower bounds,
\begin{align}
\int_{\gdgd:\gdgd \prec m \gdgd^{(i)}\mbox{\scriptsize\ for some~$i$}} e^{\gd_+} \dd \gdgd
&\geq \prod_{j = 1}^m e^{\| \gdgd^{(j)} \|}.
\label{crude}
\end{align}
Thus
\begin{align}
\lefteqn{\E\left[ \left( \cK(\gD-) \right)^r \right]} \nonumber \\
&\leq \sum_{m = 1}^r {r \brace m}\,m!\,m^{d m}
\int\!\prod_{i = 1}^m \left( e^{ - \| \gdgd^{(i)} \|} \dd \gdgd^{(i)} \right) \nonumber \\
&= \sum_{m = 1}^r {r \brace m}\,m!\,m^{d m}
\P(r_m = m;\,\|\XX^{(i)}\| \leq \gD / m\mbox{\ for $i = 1, \dots, m$}),
\label{probbound}
\end{align}
where $r_n \equiv r_n(d)$ denotes the number of remaining records at time~$n$; recall \refD{D:record}(b).

We therefore have the bound
\begin{align*}
\E\left[ \left( \cK(\gD-) \right)^r \right]
&\leq \sum_{m = 1}^r m! {r \brace m} m^{d m}\,\P(r_m = m) \\ 
&\leq \sum_{m = 1}^r m! {r \brace m} m^{d m}\,a_d^m\,m^{ - \frac{1}{d - 1} m},
\end{align*}
with $a_d$ as in the statement of the proposition.  The inequality
\begin{equation}
\label{Brightwell}
\P(r_m = m) \leq a_d^m\,m^{ - \frac{1}{d - 1} m}
\end{equation}
used here is due to Brightwell \cite[Theorem~1]{Brightwell(1992)} [who also gives a lower bound of matching form on $\P(r_m = m)$], answering a question raised by Winkler~\cite{Winkler(1985)}.

Continuing by using the simple upper bound ${r \brace m} \leq m^r / m!$, for any $\gD > 0$ we have
\begin{align}
\E\left[ \left( \cK(\gD-) \right)^r \right]
&\leq \sum_{m = 1}^r m^{r + (d -  \frac{1}{d - 1}) m}\,a_d^m \nonumber \\
&\leq a_d^r r^{(d + 2 -  \frac{1}{d - 1}) r} < \infty.
\label{approach1}
\end{align}
Since $\cK(\gD-) \uparrow \cK$ as $\gD \uparrow \infty$, the proposition follows by the monotone convergence theorem. 
\end{proof}

\begin{remark}
Except for first moments [recalling \refL{L:tight}(i) and calculating $\E \cK = 1$ by using the first sentence in the proof of \refL{L:close} with $r = 1$, integrating out~$g$ with respect to $\cL(G)$, and passing to the limit as $\gD \to \infty$] we
have not investigated whether the (also finite) moments of $K_n$ converge to the corresponding moments of~$\cK$.
\end{remark}

\begin{remark}
\label{R:do better}
(a)~When $d = 2$, the logarithmic asymptotics of the bound in \refP{P:moment bound} do not have the correct coefficient for the lead-order term.  Indeed, the moments of $\cK \equiv \cK(2)$ (see \refC{C:d=2}) are
\begin{equation}
\label{truth}
\E \cK^r 
= \sum_{r = 1}^m m! {r \brace m} 
= \frac{1}{2} \Li_{ - r}\!\left( \frac{1}{2} \right) \sim \frac{\sqrt{\pi / 2}}{\L 2} r^{r + (1 / 2)} (e \L 2)^{-r},
\end{equation}
where $\Li$ denotes polylogarithm.  

(b)~We do not know, for any $d \geq 3$, whether the logarithmic asymptotics of \refP{P:moment bound} are correct to lead order.  The bound~\eqref{crude} is rather crude.  However, see \refR{R:moment LB}(b).
\ignore{
, and the best reverse inequality we know that is tractable for use in~\eqref{EKgD-r again} is far cruder:
\[
\int_{\gdgd:\gdgd \prec m \gdgd^{(i)}\mbox{\scriptsize\ for some~$i$}} e^{ \gd_+} \dd \gdgd
\leq \sum_{j = 1}^m \int_{\gdgd:\gdgd \prec m \gdgd^{(j)}} e^{ \gd_+} \dd \gdgd
\leq m \prod_{j = 1}^m e^{m \| \gdgd^{(j)} \|}.
\]
This inequality yields only
\begin{align}
\E \cK^r
&= \lim_{\gD \to \infty} \E\left[ \left( \cK(\gD-) \right)^r \right] \nonumber \\
&\geq \sum_{m = 1}^r {r \brace m} m!\,m^{ - [(d + 1) m + 1]} \P(r_m = m).
\label{lousy}
\end{align}

The expression in~\eqref{lousy} is certainly no bigger than
\[
\sum_{m = 1}^r {r \brace m} m! = \E[(\cK(2))^r],
\]
the asymptotics of which are given by~\eqref{truth}.
}
\end{remark} 
 
\subsection{Tail asymptotics for $\cL(\cK)$}
\label{S:tail asymptotics}

The next two theorems are the main results of this subsection.  They give closely (but not perfectly) matching upper and lower bounds on lead-order logarithmic asymptotics for the tail of $\cK(d)$.  We do not presently know for any $d \geq 3$ how to close the gap.

\begin{theorem} 
\label{T:tail upper} 
Fix $d \geq 2$.  Then
\[
\P(\cK(d) \geq k) \leq \exp\left[ - \Omega\!\left( k^{1 / [d + 2 - (d - 1)^{-1}]} \right) \right]\mbox{\rm \ \ as $k \to \infty$}
\]
\end{theorem}

The exponent of~$k$ here can be written as $(d - 1) / (d^2 + d - 3)$.

\begin{theorem}
\label{T:tail lower}
Fix $d \geq 2$.  Then
\[
\P(\cK(d) \geq k) \geq \exp\left[ - O\!\left( k^{1 / (d - 1)} \right) \right]\mbox{\rm \ \ as $k \to \infty$}.
\]
\end{theorem}
\smallskip

\begin{proof}[Proof of \refT{T:tail upper}]
This follows simply from \refP{P:moment bound}, applying Markov's inequality with an integer-rounding of the optimal choice (in relation to the bound of the proposition)
\[
r = e^{-1} \left( \frac{k}{a_d} \right)^{1 / (d + 2 - \frac{1}{d - 1})}, 
\]
to wit:
\begin{align}
\hspace{-.1in}
\P(\cK(d) \geq k) 
&\leq k^{-r} \E \cK(d)^r 
\leq k^{-r} \times a_d^r r^{(d + 2 - \frac{1}{d - 1}) r} \nonumber \\
&= \exp\left[ - (1 + o(1)) \frac{(d + 2 - \frac{1}{d - 1})^{-1}}{e} \left( \frac{k}{a_d} \right)^{1 / (d + 2 - \frac{1}{d - 1})} \right].
\label{upper}
\end{align}
\end{proof}

\begin{remark}
\label{R:d=2 bound}
When $d = 2$, the truth, according 
to \refC{C:d=2}, is
\[
\P(\cK(2) \geq k) = \exp[ - (\L 2) k].
\]
\end{remark}
\smallskip

\begin{proof}[Proof of \refT{T:tail lower}]
We will establish the stronger assertion that
\begin{equation}
\label{LB+}
\P(\cK \geq k) \geq \exp\!\left[ -(1 + o(1)) (c k)^{1 / (d - 1)} \right]
\end{equation}
for $c = \frac{e}{e - 1} (d - 1)!$ by establishing it for any $c > \frac{e}{e - 1}(d - 1)!$. 
The idea of the proof is that when $G = g$ is large, $\cK = \cK_g$ will be large because $\cK_g \geq \cK_g(g-)$ and 
$\cK_g(g-)$ will be large.

Choose and fix $c \equiv c_d > \frac{e}{e - 1}(d - 1)!$.
Define $g_k := (c k)^{1 / (d - 1)}$.  Then
\begin{align}
\P(\cK \geq k) 
&\geq \int_{g_k}^{g_k + 1} \P(G \in \ddx g) \P(\cK_g \geq k) \nonumber \\
&\geq \int_{g_k}^{g_k + 1} \P(G \in \ddx g) \P(\cK_g(g-) \geq k) \nonumber \\
&= \P(G \in (g_k, g_k + 1)) \nonumber \\ 
&{} \qquad \times \int_{g_k}^{g_k + 1} \P(G \in \ddx g \mid G \in (g_k, g_k + 1)) \P(\cK_g(g-) \geq k).
\label{two factors}
\end{align}
For the first factor here, we use (with stated asymptotics as $k \to \infty$)
\begin{align*}
\P(G \in (g_k, g_k + 1))
&= \exp\!\left[ - e^{ - (g_k + 1)} \right] - \exp[ - e^{ - g_k}] \\
&= (1 + o(1)) e^{ - g_k} - (1 + o(1)) e^{ - (g_k + 1)} \\
&\sim (1 - e^{-1}) \exp\!\left[ - (c k)^{1 / (d - 1)} \right] \\
&= \exp\!\left\{ - \left[ (c k)^{1 / (d - 1)} + O(1) \right] \right\} \\
&= \exp\!\left[ - (1 + o(1)) (c k)^{1 / (d - 1)} \right].
\end{align*}
We will show that the second factor equals $1 - o(1)$.  It then follows that
\[
- \ln \P(\cK \geq k) \leq (1 + o(1)) (c k)^{1 / (d - 1)},
\]
as claimed.

It remains to show that
\[
\int_{g_k}^{g_k + 1} \P(G \in \ddx g \mid G \in (g_k, g_k + 1)) \P(\cK_g(g-) \geq k) = 1 - o(1).
\]
We will do this by applying Chebyshev's inequality to the integrand factor $\P(\cK_g(g-) \geq k)$, so to prepare we will obtain a lower bound on $\E \cK_g(g-)$ and an upper bound on $\Var \cK_g(g-)$.  

After some straightforward simplification, it follows from the case $r = 1$ (with $\gD = g$) of the calculation of
$\E [(\cK_g(\gD-))^r] = \eqref{rth moment}$ in the proof of \refL{L:close} and a change of variables (in what follows) from~$\eta$ to $v = e^{-g} e^{\eta}$ that, for any $0 < \epsilon < 1$ and uniformly for 
$g \in (g_k, g_k + 1)$, we have
\begin{align*}
\E \cK_g(g-)  
&= \exp\!\left( e^{ - g} \right) e^{ - g} \int_0^g\!\frac{\eta^{d - 1}}{(d - 1)!} e^{\eta} \exp\!\left( - e^{-g} e^{\eta} \right) \dd \eta \\
&= \exp\!\left( e^{ - g} \right) \frac{g^{d - 1}}{(d - 1)!} \int_{e^{-g}}^1 \left( 1 + \frac{\L v}{g} \right)^{d - 1} e^{-v} \dd v \\
&\geq \exp\!\left( e^{ - g} \right) \frac{g^{d - 1}}{(d - 1)!} \int_{e^{ - \epsilon g}}^1 (1 - \epsilon)^{d - 1} e^{-v} \dd v \\
&\geq \exp\!\left[ e^{ - (g_k + 1)} \right] \frac{g_k^{d - 1}}{(d - 1)!} (1 - \epsilon)^{d - 1} 
\left[ \exp\!\left( - e^{ - \epsilon g_k} \right) - e^{-1} \right] \\
&= (1 + o(1)) (1 - \epsilon)^{d - 1} \frac{c}{(d - 1)!} (1 - e^{-1}) k.
\end{align*}
Thus, uniformly for $g \in (g_k, g_k + 1)$, we have
\begin{equation}
\label{EKgg- lower}
\E \cK_g(g-) \geq (1 + o(1)) \frac{c}{(d - 1)!} (1 - e^{-1}) k. 
\end{equation}
Observe that the asymptotic coefficient of~$k$ in~\eqref{EKgg- lower} exceeds~$1$.

We next turn our attention to the second moment of $\cK_g(g-)$.  From the case $r = 2$ (with $\gD = g$) of the calculation of $\E [(\cK_g(\gD-))^r] = \eqref{rth moment}$ in the proof of \refL{L:close},
\begin{align*}
\lefteqn{\E\left[ \left( \cK_g(\gD-) \right)^2 \right]} \\
&= \exp\!\left( e^{- g} \right) \sum_{m = 1}^2 e^{ - m g} \\ 
&{} \times
\int\!\left[ \prod_{i = 1}^m \left( e^{\| \gdgd^{(i)} \|} \dd \gdgd^{(i)} \right) \right]
\exp\!\left( - e^{ - g} \sum_{\ell = 1}^m (-1)^{\ell - 1} \sum_{1 \leq i_1 < i_2 < \dots < i_{\ell} \leq m}
e^{\| \gdgd^{(i_1, \dots, i_l)} \|} \right) \\
\end{align*}
where the unlabeled integral is over vectors $\gdgd^{(1)}, \dots, \gdgd^{(m)}$ satisfying the restrictions~\eqref{restrict} with 
$\gD = g$.  

Uniformly for $g \in (g_k, g_k + 1)$, the $m = 1$ contribution to this second moment equals
\begin{align}
\E \cK_g(g-) 
&\leq \exp\!\left( e^{ - g_k} \right) \frac{(g_k + 1)^{d - 1}}{(d - 1)!} 
\int_{e^{ - (g_k + 1)}}^1 \left( 1 + \frac{\L v}{g_k} \right)^{d - 1} e^{-v} \dd v \nonumber \\
&\leq (1 + o(1))(1 - e^{-1}) \frac{g_k^{d - 1}}{(d - 1)!} \nonumber \\ 
&= (1 + o(1)) \frac{c}{(d - 1)!} (1 - e^{-1}) k.
\label{EKgg- upper}
\end{align} 
[We remark in passing that the asymptotic bounds~\eqref{EKgg- lower} and~\eqref{EKgg- upper} match.] 

The $m = 2$ contribution, call it $C_2$, is
\begin{align*}
\lefteqn{\exp\!\left( e^{- g} \right) e^{ - 2 g}} \\
&\times
\int\!e^{\| \gdgd^{(1)} \| + \| \gdgd^{(2)} \|}
\exp\!\left[ - e^{ - g} \left( e^{\| \gdgd^{(1)} \|} + e^{\| \gdgd^{(1)} \|} - e^{\| \gdgd^{(1, 2)} \|} \right) \right]
\dd \gdgd^{(1)} \dd \gdgd^{(2)}.
\end{align*}
Here we recall that $\gdgd^{(1)}$ and $\gdgd^{(2)}$ are incomparable.  If, for example, $1 \leq j \leq d - 1$ and 
$\gee^{(1)}_i := \gd^{(2)}_i - \gd^{(1)}_i > 0$ for $i = 1, \dots, j$ and 
$\gee^{(2)}_i := \gd^{(1)}_i - \gd^{(2)}_i > 0$ for $i = j + 1, \dots, d$, then,
with $\gege := \gdgd^{(1, 2)}$, 
 the integrand equals
\[
e^{2 \| \gege \| + \| \gege^{(1)} \| + \| \gege^{(2)} \|}  
\exp\!\left[ - e^{ - g} e^{\| \gege \|} \left( e^{\| \gege^{(1)} \|} + e^{\| \gege^{(2)} \|} - 1 \right) \right],
\]
where $\gege$, $\gege^{(1)}$, and $\gege^{(2)}$ are vectors of length~$d$, $j$, and $d - j$, respectively.
By this reasoning, $C_2$ equals
\begin{align*}
\lefteqn{\exp\!\left( e^{- g} \right) e^{ - 2 g} \sum_{j = 1}^{d - 1} \left\{ {d \choose j} \right.} \\
&\times
\left. \int\!e^{2 \| \gege \| + \| \gege^{(1)} \| + \| \gege^{(2)} \|}  
\exp\!\left[ - e^{ - g} e^{\| \gege \|} \left( e^{\| \gege^{(1)} \|} + e^{\| \gege^{(2)} \|} - 1 \right) \right] 
\dd \gege^{(1)} \dd \gege^{(2)} \dd \gege \right\},
\end{align*}
where the integral is over vectors $\gege, \gege^{(1)}, \gege^{(2)} \succ 0$ of lengths as previously specified, subject to the two restrictions $\| \gege \| + \| \gege^{(i)} \| \leq g$ for $i = 1, 2$.  The integral here reduces to the following three-dimensional integral:
\begin{align*}
\lefteqn{\hspace{-0.05in}\int_0^g\!\frac{\eta^{d - 1}}{(d - 1)!}
\int_0^{g - \eta} \int_0^{g - \eta} \frac{\eta_1^{j - 1}}{(j - 1)!} \frac{\eta_2^{d - j - 1}}{(d - j - 1)!}} \\
&\qquad \qquad \qquad \qquad \qquad \times e^{2 \eta + \eta_1 + \eta_2}  
\exp\!\left[ - e^{ - g} e^{\eta} \left( e^{\eta_1} + e^{\eta_2} - 1 \right) \right] 
\dd \eta_1 \dd \eta_2 \dd \eta \\
&\leq \frac{{{d - 2} \choose {d - 1 - j}}}{(d - 2)!} \int_0^g \frac{\eta^{d - 1} (g - \eta)^{d - 2}}{(d - 1)!} \\
&\qquad \qquad \quad \times \int_0^{g - \eta} \int_0^{g - \eta} e^{2 \eta + \eta_1 + \eta_2}  
\exp\!\left[ - e^{ - g} e^{\eta} \left( e^{\eta_1} + e^{\eta_2} - 1 \right) \right] 
\dd \eta_1 \dd \eta_2 \dd \eta \\
&= \frac{{{d - 2} \choose {d - 1 - j}}}{(d - 2)!} 
\int_0^g \frac{\eta^{d - 1} (g - \eta)^{d - 2}}{(d - 1)!} e^{2 \eta} \exp[e^{- (g - \eta)}] \\
&\qquad \qquad \qquad \times 
\left\{ \int_0^{g - \eta} e^{\eta_1}  
\exp\!\left[ - e^{ -(g - \eta)} e^{\eta_1} \right] 
\dd \eta_1 \right\}^2 \dd \eta. 
\end{align*}
It follows that $C_2$ is bounded above by $\exp\!\left( e^{- g} \right)$ times
\begin{align*}
\lefteqn{\hspace{-.2in}\frac{{{2 d - 2} \choose {d - 1}}}{(d - 1)! (d - 2)!} e^{ - 2 g}
\int_0^g \eta^{d - 1} (g - \eta)^{d - 2} e^{2 \eta} \exp[e^{- (g - \eta)}]} \\
&\qquad \qquad \qquad \qquad \quad \times 
\left\{ \int_0^{g - \eta} e^{\eta_1}  
\exp\!\left[ - e^{ -(g - \eta)} e^{\eta_1} \right] 
\dd \eta_1 \right\}^2 \dd \eta
\end{align*}
\begin{align*}
&= \frac{{{2 d - 2} \choose {d - 1}}}{(d - 1)! (d - 2)!} e^{ - 2 g}
\int_0^g \eta^{d - 1} (g - \eta)^{d - 2} e^{2 \eta} \exp[e^{- (g - \eta)}] \\
&\qquad \qquad \qquad \qquad \qquad \qquad \times 
\left\{ e^{g - \eta} \left( \exp\!\left[ - e^{ - (g - \eta)} \right] - e^{-1} \right) \right\}^2 \dd \eta \\
&\leq \frac{{{2 d - 2} \choose {d - 1}} (1 - e^{-1})^2}{(d - 1)! (d - 2)!}
\int_0^g \eta^{d - 1} (g - \eta)^{d - 2} \exp[e^{- (g - \eta)}] \dd \eta \\
&= \frac{{{2 d - 2} \choose {d - 1}} (1 - e^{-1})^2}{(d - 1)! (d - 2)!} g^{2 d - 2}
\int_0^1\!t^{d - 1} (1 - t)^{d - 2} \exp[e^{- (1 - t) g}] \dd t.
\end{align*}
By the dominated convergence theorem, as $g \to \infty$ we have
\[
C_2 \equiv C_2(g) \leq (1 + o(1)) \ga_d g^{2 d - 2},
\]
where
\[
\ga_d 
= (1 - e^{-1})^2 {{2 d - 2} \choose {d - 1}}  \frac{\int_0^1\!t^{d - 1} (1 - t)^{d - 2} \dd t}{{(d - 1)! (d - 2)!}}
= \left[ \frac{1 - e^{-1}}{(d - 1)!} \right]^2.
\] 
In particular, uniformly for $g \in (g_k, g_k + 1)$, we have
\begin{equation}
\label{C2g}
C_2(g) 
\leq (1 + o(1)) \left[ (1 - e^{-1}) \frac{g_k^{d - 1}}{(d - 1)!} \right]^2
\sim \left[ \frac{c}{(d - 1)!} (1 - e^{-1}) k \right]^2.
\end{equation}

We conclude from~\eqref{EKgg- upper}, \eqref{C2g}, and~\eqref{EKgg- lower} that
\[
\Var \cK_g(g-) = o(k^2), 
\] 
uniformly for $g \in (g_k, g_k + 1)$.

By Chebyshev's inequality, uniformly for $g \in (g_k, g_k + 1)$ we have
\[
\P(\cK_g(g-) \geq k) 
\geq 1 - \frac{\Var \cK_g(g-)}{\left[ \E \cK_g(g-) - k \right]^2} 
= 1 - \frac{o(k^2)}{\Theta(k^2)}
= 1 - o(1),
\]
as was to be shown.
\end{proof}

\begin{remark}
\label{R:moment LB}
(a)~When $d = 2$, the lower bound~\eqref{LB+} with $c = \frac{e}{e - 1} (d - 1)!$ gives
\[
\P(\cK \geq k) \geq \exp\!\left[ - (1 + o(1)) \frac{e}{e - 1} k \right].
\]
The logarithmic asymptotics are of the correct order (namely, linear), but the coefficient $e / (e - 1)$ is approximately $1.582$, which is roughly twice as big as the correct coefficient (recall \refR{R:d=2 bound}) $\ln 2 \doteq 0.693$.

(b)~\refT{T:tail lower} can be used to give a lower bound on the moments of~$\cK$ by starting with the lower bound
\[
\E \cK^r \geq k^r \P(\cK \geq k)
\] 
and choosing $k \equiv k(r)$ judiciously.  The result for fixed~$d$, in short, is that
\begin{equation}
\label{moment LB}
\E \cK^r \geq \exp[(1 + o(1)) (d - 1) r \L r]\mbox{\ \ as $r \to \infty$}.
\end{equation}
The lead-order terms of \refP{P:moment bound} and~\eqref{moment LB} for the logarithmic asymptotics are of the same order, but the coefficients [$d + 2 - (d - 1)^{-1}$ and $d - 1$, respectively] don't quite match.
\end{remark}

\begin{remark}
\label{fixed distance}
It may be of some interest to study the distribution of $\cK(\gD-)$ for fixed $\gD < \infty$, but 
it is then simpler for measuring the distance of a killed record from the new record to switch from $\ell^1$-distance to 
$\ell^{\infty}$ distance.  We call the analogue of $\cK(\gD-)$ for the latter distance $\tcK(\gD-)$.

The goal of this remark is to show that, for fixed~$d$ and $\gD$, as $k \to \infty$ we have 
$ - \ln \P(\tcK(\gD-) \geq k) = \Theta(k \ln k)$ (in \emph{stark} contrast to \refT{T:tail lower}); 
more precisely, we show (a)~that 
\begin{equation}
\label{tcKgD- upper}
\P(\tcK(\gD-) \geq k) \leq 
\exp\!\left[ - \frac{\ln\!\left( \frac{1}{1 - \beta} \right)}{\ln\!\left( \frac{2}{1 - \beta} \right)} (d - 1)^{-1} k \ln k + O(k) \right]
\end{equation}
with $\gb \equiv \gb_{d, \gD} := [(e^{\gD} - 1)^d + 1]^{-1} \in (0, 1)$ [so that the increasing function 
$\ln\!\left( \frac{1}{1 - \gb} \right) / \ln\!\left( \frac{2}{1 - \gb} \right)$ of~$\gb$ belongs to $(0, 1)$ as well]
and (b)~that
\begin{equation}
\label{tcKgD- lower}
\P(\tcK(\gD-) \geq k) \geq 
\exp\!\left[ - (d - 1)^{-1} k \ln k + O(k) \right] 
\end{equation}

(a)~Here is a sketch of the proof of~\eqref{tcKgD- upper}.  In order to have $\tcK(\gD-) \geq k$, the Poisson process $N_g$ in the proof of \refL{L:kd+} must have~$n$ points $\prec x$ for some $n \geq k$, and at least~$k$ of them must be maxima among these~$n$ points.  Integrating over~$g$, we find
\[
\P(\tcK(\gD-) \geq k) \leq \gb \sum_{n = k}^{\infty} (1 - \gb)^n \P(r_n \geq k)
\]
Over the event $\{r_n \ge	 k\}$ there must be some $k$-tuple of incomparable observations; thus, by finite subadditivity,
\[
\P(r_n \geq k) \leq {n \choose k} \P(r_k = k) \leq 2^n \P(r_k = k).
\]
Recall from~\eqref{Brightwell} that
\[
\P(r_k = k) \leq a_d^k k^{ - \frac{1}{d - 1} k}.
\]
Thus for any $k_0 \equiv k_0(k) \geq k$ we have
\begin{align*}
\P(\tcK(\gD-) \geq k) 
&\leq \gb \sum_{n = k}^{k_0 - 1} 2^n a_d^k k^{ - \frac{1}{d - 1} k} + \gb \sum_{n = k_0}^{\infty} (1 - \gb)^n \\ 
&\leq \gb 2^{k_0} a_d^k k^{ - \frac{1}{d - 1} k} + (1 - \gb)^{k_0}.
\end{align*}
A nearly optimal choice of $k_0$ here is to round $\left[ \ln\!\left( \frac{2}{1 - \gb} \right) \right]^{-1} (d - 1)^{-1} k \ln k$ to an integer, and this yields~\eqref{tcKgD- upper}. 

(b)~To prove~\eqref{tcKgD- lower}, we start (in similar fashion as the proof of \refT{T:tail lower}) with a study of $\tcK_g(\gD-)$.  Using the Poisson-process notation $N_g$ as we did in the proof of \refL{L:kd+}, observe that
\[
\P(\tcK_g(\gD-) \geq k) 
\geq \sum_{n = k}^{\infty}\P(N_g(A) = n) \P(r_n \geq k) \P(N_g(B) = 0),
\]
where, with $\xx := (g/d, \dots, g/d)$ and $\gDgD := (\gD, \dots, \gD)$, we set
\[
A := \{\zz:\,\xx - \gDgD \prec \zz \prec \xx\}, \qquad B := \{\zz: \xx - \gDgD \prec \zz\mbox{\ and\ } \xx \not\prec \zz\}.
\]
The random variables $N_g(A)$ and $N_g(B)$ are each Poisson distributed with respective means
\[
\gl = e^{ - g} (e^{\gD} - 1)^d, \qquad \mu = e^{-g} (e^{d \gD} - 1).
\]
Since the number $r_n$ of remaining records in dimension~$d$ has the same distribution as the total number $R_n$ of records set through time~$n$ in dimension $d - 1$, it follows that $r_n$ increases stochastically in~$n$.  Thus
\[
\P(\tcK_g(\gD-) \geq k) 
\geq \P(N_g(A) \geq k) \P(r_k = k) e^{- \mu}.
\]
If $g \leq g_k := - \L k + c$ with $ - \infty < c \equiv c_{d, \gD} < d \L(e^{\gD} - 1)$, then Chebyshev's inequality gives (uniformly for such~$g$) the result
\[
\P(N_g(A) \geq k) = 1 - o(1)
\]
as $k \to \infty$.  If $g \geq g_{k + 1}$, then
\[
e^{-\mu} \geq \exp[- e^{ - c} (e^{d \gD} - 1) (k + 1)] \geq \exp[- (e^{\gD} - 1)^{-1} (e^{d \gD} - 1) (k + 1)]. 
\]  
Further, according to \cite{Brightwell(1992)},
\[
\P(r_k = k) = \exp\!\left[ - (d - 1)^{-1} k \L k + O(k) \right]
\]
as $k \to \infty$.  To summarize our treatment thus far, uniformly for $g_{k + 1} \leq g \leq g_k$ we have
\begin{equation}
\label{tcKggD- lower}
\P(\tcK_g(\gD-) \geq k) 
\geq \exp\!\left[ - (d - 1)^{-1} k \L k + O(k) \right].
\end{equation}

But if~$G$ is distributed standard Gumbel, we also have
\begin{align}
\P(g_{k + 1} \leq G \leq g_k) 
&= e^{ - e^{ - c} k} - e^{ - e^{ - c} (k + 1)} = (1 - e^{ - e^{ - c}}) e^{ - e^{ - c} k} \nonumber \\ 
&= \exp[- O(k)].
\label{Gumbel}
\end{align}
The assertion~\eqref{tcKgD- lower} follows from~\eqref{Gumbel} and~\eqref{tcKggD- lower}.
\end{remark}

\section{Asymptotics of $\cL(\cK(d))$ as $d \to \infty$}
\label{S:d to infinity}

In this section we prove that $\cK(d)$ converges in probability (equivalently, in distribution) to~$0$  
as the dimension~$d$ tends to infinity by establishing upper (\refSS{S:d to infinity upper}) and lower (\refSS{S:d to infinity lower}) bounds, each decaying exponentially to~$0$, on $\P(\cK(d) \geq 1)$.

\subsection{Upper bound}
\label{S:d to infinity upper}

Here is the main result of this subsection, showing that the decay of $\P(\cK(d) \geq 1)$ to~$0$ is at least exponentially rapid.

\begin{theorem}
\label{T:d to infinity upper}
As $d \to \infty$ we have
\[
\P(\cK(d) \geq 1) \leq (1 + o(1))\,(0.623)^d.
\]
\end{theorem}

\begin{proof}
Passing to the limit as $\gD \to \infty$ from~\eqref{EKgD-r again} with $r = 1$, recall that for each $g \in \bbR$ we have
\begin{align*}
\E \cK_g 
&= \exp\!\left( e^{- g} \right) e^{ - g} 
\int_{\gdgd^{(1)} \succ {\bf 0}}\!e^{\| \gdgd^{(1)} \|} \exp\!\left( - e^{ - g}  e^{\| \gdgd^{(1)} \|} \right) \dd \gdgd^{(1)} \\
&= \exp\!\left( e^{- g} \right)
\int_0^{\infty}\!\frac{\eta^{d - 1}}{(d - 1)!} e^{ - (g - \eta)} \exp\!\left[ - e^{ - (g - \eta)} \right] \dd \eta.
\end{align*}
Thus for any $g > 0$ we have
\begin{equation}
\label{PK1 upper}
\P(\cK(d) \geq 1) \leq \P(|G| > g) 
+ \int_{-g}^g\!e^{ - 2 \gam} I_d(\gam) \dd \gam
\end{equation}
where the integral $I_d(\gam)$ is defined, and can be bounded for any $c > 0$, as follows:
\begin{align*}
I_d(\gam) 
&:= \int_0^{\infty}\!\frac{\eta^{d - 1}}{(d - 1)!} e^{\eta} \exp\!\left[ - e^{ - (\gam - \eta)} \right] \dd \eta \\
&\leq c^{ - (d - 1)} \int_0^{\infty}\!e^{(c + 1) \eta} \exp\!\left[ - e^{ - (\gam - \eta)} \right] \dd \eta \\
&\leq c^{ - (d - 1)} \int_{- \infty}^{\infty}\!e^{(c + 1) \eta} \exp\!\left[ - e^{ - (\gam - \eta)} \right] \dd \eta \\
&= c^{ - (d - 1)} \Gamma(c + 1) \exp[(c + 1) \gam].
\end{align*}
Returning to~\eqref{PK1 upper}, for $c > 1$ we find
\begin{align*}
\P(\cK(d) \geq 1) 
&\leq \P(|G| > g) 
+ c^{ - (d - 1)} \Gamma(c + 1) \int_{-g}^g\!e^{ (c - 1) \gam} \dd \gam \\
&\leq \P(|G| > g) + c^{ - (d - 1)} \Gamma(c + 1) (c - 1)^{-1} e^{(c - 1) g}.
\end{align*}

As $g \to \infty$, this last bound has the asymptotics
\[
(1 + o(1)) e^{-g} + c^{ - (d - 1)} \Gamma(c + 1) (c - 1)^{-1} e^{(c - 1) g}.
\]
Thus, to obtain an approximately optimal bound we choose 
\[
g = c^{-1} \L\left[ \frac{c^{d - 1}}{\Gamma(c + 1)} \right].
\]
This choice gives the following asymptotics as $d \to \infty$:
\[
\P(\cK(d) \geq 1) \leq (1 + o(1)) \frac{c}{c - 1} \left[ \frac{c^{d - 1}}{\Gamma(c + 1)} \right]^{- c^{-1}}.
\]
The optimal choice of~$c$ for large~$d$ minimizes $c^{- c^{-1}}$, leading to $c = e$ and
\[
\P(\cK(d) \geq 1) \leq (1 + o(1))\,(0.623)^d,
\] 
as claimed, since $\exp[- e^{-1}] < 0.623$.
\end{proof}

\subsection{Lower bound}
\label{S:d to infinity lower}

Here is the main result of this subsection, showing that the decay of $\P(\cK(d) \geq 1)$ to~$0$ is at most exponentially rapid.

\begin{theorem}
\label{T:d to infinity lower}
For every $d \geq 1$ we have
\[
\P(\cK(d) \geq 1) \geq 4^{- d}.
\]
\end{theorem}

\begin{proof}
We use Poisson process notation as in the proof of \refL{L:kd+}.
Let ${\bf 1} := (1, \dots, 1) \in \bbR^d$. 
Given $g \in \bbR$ and $c > 0$, let $\xx(g) = \frac{g}{d} {\bf 1}$ and
\[
S_c := \{ \zz \in \bbR^d:\zz \succ \xx - c {\bf 1},\,\zz \not\succ \xx \},
\]
and consider the subset
\[
S'_c := \{ \zz \in \bbR^d:\xx - c {\bf 1} \prec \zz \prec \xx \}
\]
of $S_c$.  Then
\begin{align}
\P(\cK_g \geq 1) 
&\geq \P(N_g(S'_c) = 1 = N_g(S_c)) \nonumber \\
&= \P(N_g(S'_c) = 1,\,N_g(S_c - S'_c) = 0) \nonumber \\
&= e^{ - \gl} \frac{\gl^1}{1!} \times \exp\!\left( - \int_{\zz \in S_c - S'_c} e^{ - z_+} \dd \zz \right)
\label{none here}
\end{align}
with
\[
\gl = \int_{\zz \in S'_c} e^{- z_+} \dd \zz 
= \prod_{i = 1}^d \left\{ \exp\left[ - (x_i(g) - c) \right] - \exp\left[ - x_i(g) \right] \right\} = e^{-g} \left( e^c - 1 \right)^d.
\]
The integral in~\eqref{none here} equals the difference of the integrals over $S_c$ and over $S'_c$, namely,
\[
\int_{\zz \succ \xx - c {\bf 1}} e^{ - z_+} \dd \zz - \int_{\zz \succ \xx} e^{ - z_+} \dd \zz - \gl
= e^{ - (g - c d)} - e^{-g} - \gl.
\]
So
\[
\P(\cK_g \geq 1) \geq \left( e^c - 1 \right)^d e^{-g} \exp\!\left[ - e^{-g} (e^{c d} - 1)  \right]
\]
and hence
\begin{align*}
\P(\cK(d) \geq 1) 
&\geq \left( e^c - 1 \right)^d \int_{- \infty}^{\infty} e^{ - 2 g} \exp\!\left[ - e^{-g} e^{c d} \right] \dd g \\
&= \left( e^c - 1 \right)^d e^{ - 2 c d} = \left( \frac{e^c - 1}{e^{2 c}} \right)^d.
\end{align*}
The optimal choice of~$c$ is then $c = \ln 2$, yielding the claimed result.
\end{proof}

\begin{remark}
In the same spirit as their Conjecture~2.3 and the discussion following it, the authors of~\cite{Fillgenerating(2018)} might have put forth the closely related conjecture that $\P(K_n = 0\!\mid\!K_n \geq 0)$ has a limit $q_d$ and that $q_d \to 1$ as $d \to \infty$, with the additional suggestion that perhaps $q_d = 1 - d^{-1}$ for every $d \geq 2$.  In light of Theorems~\ref{T:main} and \ref{T:d to infinity upper}--\ref{T:d to infinity lower} we see that the related conjecture is true, but the additional suggestion is not.
\end{remark}

\section{Dimension $d = 2$}
\label{S:d=2}
Here is the main result of this section.

\begin{corollary}
\label{C:d=2}
Adopt the setting and notation of \refT{T:main} with $d = 2$.  Then
\smallskip 

\noindent
{\rm (i)} For each $g \in \bbR$, the law of $\cK_g$ is the mixture $\E \mbox{\rm Poisson}(\gL_g)$, where 
$\cL(\gL_g)$ is the conditional distribution of $g - G$ given 
$g - G > 0$ with $G \sim$ {\rm standard Gumbel}.
\smallskip

\noindent
{\rm (ii)} $\cK = \cK(2)$ has the same distribution as $\cG - 1$, where $\cG$ {\rm (}with support $\{1, 2, \dots\}${\rm )} has the {\rm Geometric}$(1/2)$ distribution.
\end{corollary}

\begin{proof} 
We first show how~(ii) follows from~(i).  If~(i) holds, then the law of $\cK$ is the mixture 
$\E \mbox{\rm Poisson}(\gL)$, where $\cL(\gL)$ is the mixture $\E \cL(\gL_G)$ with $G \sim$ standard Gumbel.
Observe that the density of $\gL_g$ is
\begin{equation}
\label{gLg density}
\gl \mapsto {\bf 1}(\gl > 0) e^{\gl - g} \exp\!\left[ - e^{-g} (e^{\gl} - 1) \right],
\end{equation}
so the density of~$\gL$ is
\[
\gl \mapsto \int_{- \infty}^{\infty} e^{\gl - 2 g} \exp\!\left[ - e^{\gl - g} \right] \dd g = e^{ - \gl};
\] 
that is, $\gL$ has the Exponential$(1)$ distribution.  But then $\cL(\cK)$ is the law of $\cG - 1$, as follows either computationally:\ for $k = 0, 1, \dots$ we have
\[
\P(\cK = k) = \int_0^{\infty}\!e^{ - w} \frac{w^k}{k!} e^{ - w} \dd w = 2^{ - (k + 1)};
\]
or by the probabilistic argument that $\cK$ has the same distribution as the number of arrivals in one Poisson counting process with unit rate prior to the (Exponentially distributed) epoch of first arrival in an independent copy of the process, which (using symmetry of the two processes and the memoryless property of the Exponential distribution) has the same distribution as $\cG - 1$.

We now proceed to prove~(i).
It is easy to see that, with probability~$1$, the Poisson point process (call it $N_g$) described in \refT{T:main} will have an infinite number of maxima, but with no accumulation points in either of the coordinates, and a finite number of maxima dominated by $\xx = (x, y) := (g / 2, g / 2)$.  [It is worth noting that the argument to follow is unchanged if $\xx$ is changed to any other point $(x, y)$ with $x + y = g$.]  Thus, over the event $\{\cK_g = k\}$ we can list the locations of the maxima $\XX_i = (X_i, Y_i)$, in order from northwest to southeast (\ie,\ in increasing order of first coordinate and decreasing order of second coordinate), as 
$ \ldots, \XX_0, \XX_1, \ldots, \XX_k, \XX_{k + 1}, \ldots$, where $\XX_1, \ldots, \XX_k$ are the maxima dominated by~$\xx$.  
Given
\begin{align}
\label{xrestrict}
\lefteqn{\hspace{.65in}- \infty < x_0 < x_1 < \cdots < x_k < x < x_{k + 1} < \infty} \\
\label{yrestrict}
&{}&\mbox{and\ }\infty > y_0 > y > y_1 > \cdots > y_k > y_{k + 1} > - \infty, 
\end{align}
let~$S$ denote the following disjoint union of rectangular regions:
\[
S = \bigcup_{i = 1}^k [(x_{i - 1}, x_i) \times (y_i, \infty)] 
\,\cup\,[(x_k, x) \times (y_{k + 1}, \infty)]\,\cup\,[(x, \infty) \times (y_{k + 1}, y)].
\]
Then (by the same sort of reasoning as in \cite[Proposition~3.1]{Fillbreaking(2021)}),
for $k = 0, 1, \ldots$ and $\xx_0, \dots, \xx_k$ satisfying~\eqref{xrestrict}--\eqref{yrestrict}, and introducing
abbreviations
\[
\sum_j^k := \sum_{i = j}^k (e^{ - x_{i - 1}} - e^{ - x_i}) e^{- y_i}, \qquad \sum^k := \sum_1^k,
\]
we have
\begin{align}
\lefteqn{\hspace{-.05in}\P(\cK_g = k;\,\XX_i \in \ddx \xx_i\mbox{\ for $i = 0, \ldots, k + 1$})} \nonumber \\
&= \P(N_g(\ddx \xx_i) = 1\mbox{\ for $i = 0, \dots, k + 1$}; N_g(S) = 0) \nonumber \\
&= \left[ \prod_{i = 0}^{k + 1} (e^{ - (x_i + y_i)} \dd \xx_i) \right] 
\times \exp\!\left[- \int_{\zz \in S}\!e^{ - z_+} \dd z \right] \nonumber \\
&= \exp\!\left[ - \sum_{i = 0}^{k + 1} (x_i + y_i) \right]\,
\exp\!\left\{ e^{ - g} - \left[ \sum^k {} + e^{ - (x_k + y_{k + 1})} \right] \right\} \dd \xx_0 \cdots \dd \xx_{k + 1}.
\label{joint}
\end{align}
To calculate $\P(\cK_g = k)$, we need to integrate this last expression over all $\xx_0, \ldots, \xx_{k + 1}$ satisfying \eqref{xrestrict}--\eqref{yrestrict}.

Since $x_{k + 1}$ and $y_0$ appear only in the first of the two factors in~\eqref{joint}, we integrate them out first to obtain
\begin{align}
\label{integral result}
\lefteqn{e^{ - g} \int \exp\!\left( - \sum_{i = 0}^k x_i \right) \exp\!\left( - \sum_{i = 1}^{k + 1} y_i \right)} \\
&{} \qquad \times 
\exp\!\left\{ e^{ - g} - \left[ \sum^k {} + e^{ - (x_k + y_{k + 1})} \right] \right\} 
\dd x_0 \dd \xx_1 \cdots \dd \xx_k \dd y_{k + 1},
\nonumber
\end{align}
where the integral is over all $x_0, \xx_1, \ldots \xx_k, y_{k + 1}$ satisfying
\begin{align*}
\lefteqn{\hspace{.5in}- \infty < x_0 < x_1 < \cdots < x_k < x} \\
&{}&\hspace{-.8in}\mbox{and\ }y > y_1 > \cdots > y_k > y_{k + 1} > - \infty. 
\end{align*}

Next we integrate out $x_0$.  For this we use the calculation
\[
\int_{- \infty}^{x_1} e^{ - x_0} \exp\!\left( - e^{ - y_1} e^{ - x_0} \right) \dd x_0 
= e^{y_1} \exp\!\left( - e^{ - x_1} e^{ - y_1} \right).
\]
So when we integrate out $x_0$ we get
\begin{align*}
\lefteqn{e^{ - g} \int \exp\!\left( - \sum_{i = 1}^k x_i \right) \exp\!\left( - \sum_{i = 2}^{k + 1} y_i \right)} \\
&{} \qquad \times 
\exp\!\left\{ e^{ - g} - \left[ \sum_2^k {} + e^{ - (x_k + y_{k + 1})} \right] \right\}
\dd \xx_1 \cdots \dd \xx_k \dd y_{k + 1},
\end{align*}
where the integral is over all $\xx_1, \ldots \xx_k, y_{k + 1}$ satisfying
\begin{align*}
\lefteqn{\hspace{.5in}- \infty < x_1 < \cdots < x_k < x} \\
&{}&\hspace{-.8in}\mbox{and\ }y > y_1 > \cdots > y_k > y_{k + 1} > - \infty. 
\end{align*}
Continuing in this fashion, after integrating out $x_1, \dots, x_k$ we get
\[
e^{ - g} \int \exp\!\left\{ e^{ - g} - e^{ - (x + y_{k + 1})} \right\} \dd y_1 \cdots \dd y_k \dd y_{k + 1},
\]
where the integral is over all $y_1, \dots, y_{k + 1}$ satisfying
\[
y > y_1 > \cdots > y_k > y_{k + 1} > - \infty.
\]

Next we integrate out $y_1, \dots, y_k$ and change the remaining variable name from $y_k$ to~$\eta$ to find
\begin{align}
\P(\cK_g = k) 
&= \frac{e^{ - g}}{k!} \int_{- \infty}^y (y - \eta)^k \exp\!\left\{ e^{ - g} - e^{ - (x + \eta)} \right\} \dd \eta \nonumber \\
&= \int_0^{\infty} e^{ - \gl} \frac{\gl^k}{k!} e^{\gl - g} \exp\!\left[ - e^{ - g} (e^{\gl} - 1) \right] \dd \gl \nonumber \\
&= \int_0^{\infty} e^{ - \gl} \frac{\gl^k}{k!} \P(\gL_g \in \ddx \gl),
\label{the end}
\end{align} 
where we have recalled~\eqref{gLg density} at~\eqref{the end}.  This completes the proof of~(i) and thus the proof of the corollary.
\end{proof}

\medskip
\begin{acks}
We thank Ao Sun for providing a simplified proof of \refL{L:tight}(i), and Daniel~Q.\ Naiman and Ao Sun for valuable assistance in producing the three figures.  We thank Svante Janson for helpful discussions about the details of this paper.
We are also grateful for discussions with Persi Diaconis, Hsien-Kuei Hwang, Daniel~Q.\ Naiman, Robin Pemantle, Ao Sun, and Nicholas Wormald.  Last but not least, we thank two anonymous reviewers for many helpful suggestions.
\end{acks}

\bibliography{records.bib}
\bibliographystyle{plain} 
\end{document}

%% file: breaking_figure.tex
\begin{adjustbox}{max totalsize={.9\textwidth}{.7\textheight},center}
\begin{tikzpicture}[scale=7]

\begin{pgfonlayer}{background layer}

%
%

%
%
\fill[fill=gray!10]    (0,.81) -- (0,1) -- (.12,1) -- (.12,.81);
\fill[fill=gray!10]    (.12,.70) -- (.12,1.) -- (.23,1) -- (.23,.70);
\fill[fill=gray!10]    (.23,.58) -- (.23,1.) -- (.33,1) -- (.33,.58);
\fill[fill=gray!10]    (.33,.44) -- (.33,1) -- (.51,1) -- (.51,.44);
\fill[fill=gray!10]    (.51,.25) -- (.51,1.) -- (.62,1) -- (.62,.25);
\fill[fill=gray!10]    (.62,.14) -- (.62,1.) -- (.82,1.) -- (.82,.14);
\fill[fill=gray!10]    (.82,0) -- (.82,1.) -- (1,1) -- (1.,0.);

%
%

%
%
\filldraw [dgreen]
(.12,.81) circle (.25pt)
;
\filldraw[red]
(.23,.70) circle (.25pt)
(.33,.58) circle (.25pt)
(.51,.44) circle (.25pt)
;
\filldraw[dgreen]
(.62,.25) circle (.25pt)
(.82,.14) circle (.25pt)
;

\draw[dashed,color=black, thick](0.58,.75)--(.12,.75);
\draw[dashed,color=black, thick](0.58,.75)--(.58,.25);
%

%
%
\draw[thick,color=black]
(.0,.81)--(.12,.81) --
(.12,.70)--(.23,.70) --
(.23,.58)--(.33,.58) --
(.33,.44)-- (.51,.44) --
(.51,.25)-- (.62,.25) --
(.62,.14) -- (.82,.14) -- (0.82, 0)
;

%
%
%
%

%
%
\def\dx{0.}
\def\dy{.04}
\draw (0+\dx,-\dy) node[color=black] {\!\!\!\!\footnotesize $0$};

%
%

%
%
\filldraw[green]
(.58,.75) circle (.25pt)
;
\end{pgfonlayer}

\begin{pgfonlayer}{foreground layer}
\draw[thick,color=black] (0,1)--(0,0)--(1,0);
\end{pgfonlayer}

\end{tikzpicture}
\end{adjustbox}

%% file: figure1.tex
\begin{adjustbox}{max totalsize={.9\textwidth}{.7\textheight},center}
\begin{tikzpicture}[scale=7]
\draw[color=black,solid,thick](-0.2575755771623076,-1.9142739173960515)--(0.5362581085663101,-1.455953825248841);
\draw[color=black,solid,thick](0.5362581085663101,-1.455953825248841)--(-0.43444147459620586,-0.8955201596077051);
\draw[color=black,solid,thick](0.5362581085663101,-1.455953825248841)--(0.5362581085663101,-1.636381118407808);
\draw[color=black,solid,thick](-0.2575755771623076,-1.9142739173960515)--(-1.942852250259816,-0.941278976524208);
\draw[color=black,solid,thick](-1.942852250259816,-0.941278976524208)--(-1.3624390450436483,-0.60617725625144);
\draw[color=black,solid,thick](-1.942852250259816,-0.941278976524208)--(-1.942852250259816,-1.121706269683175);
\draw[color=black,solid,thick](-0.2575755771623076,-1.9142739173960515)--(-0.2575755771623076,-2.0947012105550185);
\draw[color=black,solid,thick](-0.2575755771623076,-2.0947012105550185)--(0.5362581085663101,-1.636381118407808);
\draw[color=black,solid,thick](-0.2575755771623076,-2.0947012105550185)--(-1.942852250259816,-1.121706269683175);
\draw[color=black,solid,thick](1.1865525987720142,0.06663339463679119)--(1.6313637441374893,0.32344522915209406);
\draw[color=black,solid,thick](1.6313637441374893,0.32344522915209406)--(1.4227264509264796,0.4439020265504653);
\draw[color=black,solid,thick](1.6313637441374893,0.32344522915209406)--(1.6313637441374893,-0.4304951349833537);
\draw[color=black,solid,thick](1.1865525987720142,0.06663339463679119)--(0.2513166656967584,0.6065921123536011);
\draw[color=black,solid,thick](0.2513166656967584,0.6065921123536011)--(0.3117430076602592,0.6414792771523726);
\draw[color=black,solid,thick](0.2513166656967584,0.6065921123536011)--(0.2513166656967584,0.36910122321068184);
\draw[color=black,solid,thick](1.1865525987720142,0.06663339463679119)--(1.1865525987720142,-0.6873069694986566);
\draw[color=black,solid,thick](1.1865525987720142,-0.6873069694986566)--(1.6313637441374893,-0.4304951349833537);
\draw[color=black,solid,thick](1.1865525987720142,-0.6873069694986566)--(0.9978411846641773,-0.5783543837643418);
\draw[color=black,solid,thick](-0.10188361164666793,2.115297774717315)--(0.6961278110622334,2.5760298844346954);
\draw[color=black,solid,thick](0.6961278110622334,2.5760298844346954)--(0.0,2.977939463541861);
\draw[color=black,solid,thick](0.6961278110622334,2.5760298844346954)--(0.6961278110622334,1.9561916149441554);
\draw[color=black,solid,thick](-0.10188361164666793,2.115297774717315)--(-0.7980114227089014,2.5172073538244804);
\draw[color=black,solid,thick](-0.7980114227089014,2.5172073538244804)--(0.0,2.977939463541861);
\draw[color=black,solid,thick](-0.7980114227089014,2.5172073538244804)--(-0.7980114227089014,0.5670905271157696);
\draw[color=black,solid,thick](-0.10188361164666793,2.115297774717315)--(-0.10188361164666793,0.16518094800860406);
\draw[color=black,solid,thick](-0.10188361164666793,0.16518094800860406)--(0.2513166656967584,0.36910122321068184);
\draw[color=black,solid,thick](-0.10188361164666793,0.16518094800860406)--(-0.7980114227089014,0.5670905271157696);
\draw[color=black,solid,thick](1.0550507459120508,1.4061056218090915)--(1.3519777228080998,1.577536491849687);
\draw[color=black,solid,thick](1.3519777228080998,1.577536491849687)--(0.6961278110622334,1.9561916149441554);
\draw[color=black,solid,thick](1.3519777228080998,1.577536491849687)--(1.3519777228080998,0.6654418574334658);
\draw[color=black,solid,thick](1.0550507459120508,1.4061056218090915)--(0.39920083416618446,1.7847607449035596);
\draw[color=black,solid,thick](0.39920083416618446,1.7847607449035596)--(0.6961278110622334,1.9561916149441554);
\draw[color=black,solid,thick](0.39920083416618446,1.7847607449035596)--(0.39920083416618446,0.8726661104873383);
\draw[color=black,solid,thick](1.0550507459120508,1.4061056218090915)--(1.0550507459120508,0.4940109873928702);
\draw[color=black,solid,thick](1.0550507459120508,0.4940109873928702)--(1.3519777228080998,0.6654418574334658);
\draw[color=black,solid,thick](1.0550507459120508,0.4940109873928702)--(0.39920083416618446,0.8726661104873383);
\draw[color=black,solid,thick](-0.6478619551086559,-0.242005219307163)--(-0.201089630713166,0.01593890244904783);
\draw[color=black,solid,thick](-0.201089630713166,0.01593890244904783)--(-0.9156667206481585,0.4285001776797551);
\draw[color=black,solid,thick](-0.201089630713166,0.01593890244904783)--(-0.201089630713166,-0.42984854104428194);
\draw[color=black,solid,thick](-0.6478619551086559,-0.242005219307163)--(-1.3624390450436483,0.17055605592354428);
\draw[color=black,solid,thick](-1.3624390450436483,0.17055605592354428)--(-0.9156667206481585,0.4285001776797551);
\draw[color=black,solid,thick](-1.3624390450436483,0.17055605592354428)--(-1.3624390450436483,-0.60617725625144);
\draw[color=black,solid,thick](-0.6478619551086559,-0.242005219307163)--(-0.6478619551086559,-1.0187385314821473);
\draw[color=black,solid,thick](-0.6478619551086559,-1.0187385314821473)--(-0.43444147459620586,-0.8955201596077051);
\draw[color=black,solid,thick](-0.6478619551086559,-1.0187385314821473)--(-1.3624390450436483,-0.60617725625144);
\draw[color=black,solid,thick](1.190670477917273,-1.502833114172986)--(2.3396890424484713,-0.8394469366370213);
\draw[color=black,solid,thick](2.3396890424484713,-0.8394469366370213)--(1.6313637441374893,-0.4304951349833537);
\draw[color=black,solid,thick](2.3396890424484713,-0.8394469366370213)--(2.3396890424484713,-1.3508200984776428);
\draw[color=black,solid,thick](1.190670477917273,-1.502833114172986)--(-0.43444147459620586,-0.5645742909260506);
\draw[color=black,solid,thick](-0.43444147459620586,-0.5645742909260506)--(-0.201089630713166,-0.42984854104428194);
\draw[color=black,solid,thick](-0.43444147459620586,-0.5645742909260506)--(-0.43444147459620586,-0.8955201596077051);
\draw[color=black,solid,thick](1.190670477917273,-1.502833114172986)--(1.190670477917273,-2.0142062760136077);
\draw[color=black,solid,thick](1.190670477917273,-2.0142062760136077)--(2.3396890424484713,-1.3508200984776428);
\draw[color=black,solid,thick](1.190670477917273,-2.0142062760136077)--(0.5362581085663101,-1.636381118407808);
\draw[color=black,solid,thick](1.0383416475245055,0.40267039049296427)--(1.4227264509264796,0.6245950602094955);
\draw[color=black,solid,thick](1.4227264509264796,0.6245950602094955)--(1.3519777228080998,0.6654418574334658);
\draw[color=black,solid,thick](1.4227264509264796,0.6245950602094955)--(1.4227264509264796,0.4439020265504653);
\draw[color=black,solid,thick](1.0383416475245055,0.40267039049296427)--(0.3117430076602592,0.8221723108114029);
\draw[color=black,solid,thick](0.3117430076602592,0.8221723108114029)--(0.39920083416618446,0.8726661104873383);
\draw[color=black,solid,thick](0.3117430076602592,0.8221723108114029)--(0.3117430076602592,0.6414792771523726);
\draw[color=black,solid,thick](1.0383416475245055,0.40267039049296427)--(1.0383416475245055,0.22197735683393405);
\draw[color=black,solid,thick](1.0383416475245055,0.22197735683393405)--(1.4227264509264796,0.4439020265504653);
\draw[color=black,solid,thick](1.0383416475245055,0.22197735683393405)--(0.3117430076602592,0.6414792771523726);
\draw[color=black,solid,thick](0.5269856093814939,-0.33375350191070674)--(0.9978411846641773,-0.061904908771813316);
\draw[color=black,solid,thick](0.9978411846641773,-0.061904908771813316)--(0.2513166656967584,0.36910122321068184);
\draw[color=black,solid,thick](0.9978411846641773,-0.061904908771813316)--(0.9978411846641773,-0.5783543837643418);
\draw[color=black,solid,thick](0.5269856093814939,-0.33375350191070674)--(-0.9156667206481585,0.4991622091789538);
\draw[color=black,solid,thick](-0.9156667206481585,0.4991622091789538)--(-0.7980114227089014,0.5670905271157696);
\draw[color=black,solid,thick](-0.9156667206481585,0.4991622091789538)--(-0.9156667206481585,0.4285001776797551);
\draw[color=black,solid,thick](0.5269856093814939,-0.33375350191070674)--(0.5269856093814939,-0.8502029769032352);
\draw[color=black,solid,thick](0.5269856093814939,-0.8502029769032352)--(0.9978411846641773,-0.5783543837643418);
\draw[color=black,solid,thick](0.5269856093814939,-0.8502029769032352)--(-0.201089630713166,-0.42984854104428194);
\draw[->,color=black,solid,thick](-1.942852250259816,-1.121706269683175)--(-2.525707925337761,-1.4582181505881275);
\draw[->,color=black,solid,thick](2.3396890424484713,-1.3508200984776428)--(3.041595755183013,-1.7560661280209358);
\draw[->,color=black,solid,thick](0.0,2.977939463541861)--(0.0,3.8713213026044193);
\filldraw [dgreen]
(-0.2575755771623076,-1.9142739173960515) circle (1.25pt)
(1.1865525987720142,0.06663339463679119) circle (1.25pt)
(-0.10188361164666793,2.115297774717315) circle (1.25pt)
(1.0550507459120508,1.4061056218090915) circle (1.25pt)
(-0.6478619551086559,-0.242005219307163) circle (1.25pt)
(1.190670477917273,-1.502833114172986) circle (1.25pt)
(1.0383416475245055,0.40267039049296427) circle (1.25pt)
(0.5269856093814939,-0.33375350191070674) circle (1.25pt)
;\filldraw [violet]
(-0.43444147459620586,-0.8955201596077051) circle (1.25pt)
(0.3117430076602592,0.6414792771523726) circle (1.25pt)
(0.2513166656967584,0.36910122321068184) circle (1.25pt)
(0.9978411846641773,-0.5783543837643418) circle (1.25pt)
(0.39920083416618446,0.8726661104873383) circle (1.25pt)
(-0.201089630713166,-0.42984854104428194) circle (1.25pt)
;\filldraw [violet]
(0.6961278110622334,1.9561916149441554) circle (1.25pt)
(1.3519777228080998,0.6654418574334658) circle (1.25pt)
(1.4227264509264796,0.4439020265504653) circle (1.25pt)
(1.6313637441374893,-0.4304951349833537) circle (1.25pt)
(-0.7980114227089014,0.5670905271157696) circle (1.25pt)
(-0.9156667206481585,0.4285001776797551) circle (1.25pt)
(-1.3624390450436483,-0.60617725625144) circle (1.25pt)
(0.5362581085663101,-1.636381118407808) circle (1.25pt)
;\filldraw [violet]
(-1.942852250259816,-1.121706269683175) circle (1.25pt)
(2.3396890424484713,-1.3508200984776428) circle (1.25pt)
(0.0,2.977939463541861) circle (1.25pt)
;\draw (-2.4,-1.6) node[color=black, scale=3.5]{ $x_1$};
\draw (2.9,-1.9) node[color=black, scale=3.5]{ $x_2$};
\draw (0.2,3.5) node[color=black, scale=3.4]{ $x_3$};
\draw[->,color=black,dashed,ultra thick](-0.201089630713166,-0.42984854104428194)--(-1.933140438282043,-1.429848541044282);
\draw[->,color=black,dashed,ultra thick](-0.201089630713166,-0.42984854104428194)--(1.5309611768557114,-1.429848541044282);
\draw[->,color=black,dashed,ultra thick](-0.201089630713166,-0.42984854104428194)--(-0.201089630713166,1.5701514589557182);
%
%
\draw (-.2+.10,-.43+.08) node[color=black, scale=3.4]{ $g$};

\end{tikzpicture}
\end{adjustbox}

%% file: poisson.tex
\begin{adjustbox}{max totalsize={.9\textwidth}{.7\textheight}}
	\begin{tikzpicture}
	\filldraw[gray!10] (1.6, 1.6) -- (4, 1.6) -- (4, 4) -- (1.6, 4) -- cycle;
	\draw (-2.7, 0) -- (4, 0);
	\draw (0, -2.7) -- (0, 4);
	\draw[dashed] (-2.5, 1.6) -- (1.6, 1.6) node[scale=0.8, anchor=south west] {$\xx$} -- (1.6, -2.5);
	\draw[dashed] (1.6, 4) -- (1.6, 1.6) -- (4, 1.6);
	\filldraw [green] (1.6,1.6) circle [radius=1.5pt];
	\filldraw [black] 
	(-1.184,2.57) circle [radius=1pt]
	(-1.184,3.657) circle [radius=1pt]
	(-0.861,3.326) circle [radius=1pt]
	(-1.5,2.952) circle [radius=1pt]
	(-1.125,1.757) circle [radius=1pt]
	(-0.801,1.512) circle [radius=1pt]
	(-0.236,1.758) circle [radius=1pt]
	(-0.654,1.857) circle [radius=1pt]
	(-0.588,2.607) circle [radius=1pt]
	
	(-1.184,-1.57) circle [radius=1pt]
	(-0.432,0.872) circle [radius=1pt]
	(-1.194,0.857) circle [radius=1pt]
	(-0.821,0.326) circle [radius=1pt]
	(-1.5,-.049) circle [radius=1pt]
	(-1.165,-0.227) circle [radius=1pt]
	(-0.821,-1.412) circle [radius=1pt]
	(-0.256,-1.658) circle [radius=1pt]
	(-0.644,-1.857) circle [radius=1pt]
	(-2.307, -0.578) circle [radius=1pt]
	(-0.276,-0.658) circle [radius=1pt]
	(-0.344,-0.857) circle [radius=1pt]
	(-0.578,-0.907) circle [radius=1pt]

	(.765,-0.327) circle [radius=1pt]
	(.821,-1.312) circle [radius=1pt]
	(0.356,-1.758) circle [radius=1pt]
	(0.444,-1.627) circle [radius=1pt]
	(1.12,.149) circle [radius=1pt]
	(.365,-2.27) circle [radius=1pt]
	(-2.0,-1.412) circle [radius=1pt]
	(0.256,-1.558) circle [radius=1pt]
	(0.644,-1.857) circle [radius=1pt]
	;

	\filldraw [black] 
	(0.411,1.711) circle [radius=1pt]
	(2.652,0.432) circle [radius=1pt]
	(3.011,-0.363) circle [radius=1pt]
	(2.924,-1.345) circle [radius=1pt]
	(2.224,-0.892) circle [radius=1pt]
	(-2.012,2.2) circle [radius=1pt]
	(-0.801,1.512) circle [radius=1pt]
	(-1.825,2.011) circle [radius=1pt]
	(1.829,0.600) circle [radius=1pt]
	;
	\filldraw [dgreen]
	(0.555,2.414) circle [radius=1pt]
	(0.668,2.089) circle [radius=1pt]
	(0.355,3.214) circle [radius=1pt]
	(0.099,3.724) circle [radius=1pt]
	(2.368,0.82) circle [radius=1pt]
	(3.156,0.713) circle [radius=1pt]
	(3.524,-0.642) circle [radius=1pt]
	(-0.452,3.872) circle [radius=1pt]
	;
	
	\filldraw [red] 
	(0.98,1.429) node[anchor=north east, black, scale=0.6] {$\mathbf{Z}^{(1)}$}  circle [radius=1pt]
	(1.38,1.049) node[anchor=north east, black, scale=0.6] {$\mathbf{Z}^{(2)}$} circle [radius=1pt];
\end{tikzpicture}
\end{adjustbox}